\documentclass[twoside,11pt,reqno]{amsart}
\usepackage{amsmath,amssymb,amscd,mathrsfs,epic,wasysym,latexsym,tikz,mathrsfs,cite,hyperref}
\usepackage{pb-diagram}
\usepackage[matrix,arrow]{xy}
\usepackage{diagbox}
\usepackage{lscape}

\usepackage{comment}
\usepackage{hyperref}

\excludecomment{answer}      %<----- Comment to show details
\usepackage{ stmaryrd }

\usepackage[enableskew]{youngtab}
\usepackage{ytableau}
\usepackage{color}
\usepackage{enumerate}
\usepackage[margin=1in]{geometry}

\usetikzlibrary{shapes,snakes,calendar,matrix,backgrounds,folding}

\usepackage[font=small,labelfont=bf]{caption}

\usepackage{mathtools}
\DeclarePairedDelimiter\floor{\lfloor}{\rfloor}

\usepackage{centernot}
\usepackage{mathtools}
\usepackage{stmaryrd}

\makeatletter
\newcommand{\xMapsto}[2][]{\ext@arrow 0599{\Mapstofill@}{#1}{#2}}
\def\Mapstofill@{\arrowfill@{\Mapstochar\Relbar}\Relbar\Rightarrow}
\makeatother

\usepackage{young}

\makeatletter

\usepackage{tikz}
\usetikzlibrary{decorations.pathreplacing,decorations.pathmorphing,calc}

\numberwithin{equation}{section}
\newtheorem{Proposition}[equation]{Proposition}
\newtheorem{Lemma}[equation]{Lemma}
\newtheorem{Theorem}[equation]{Theorem}
\newtheorem{Corollary}[equation]{Corollary}
\theoremstyle{definition}  %% makes all of the theorem environments which follow appear in \rm
\newtheorem{Definition}[equation]{Definition}
\newtheorem{Remark}[equation]{Remark}

\newtheorem{Example}[equation]{Example}

\let\<\langle
\let\>\rangle

\newcommand\Comment[2][\relax]{\space\par\medskip\noindent%
   \fbox{\begin{minipage}{\textwidth}\textbf{Comment\ifx\relax#1\else---#1\fi}\newline%
        #2\end{minipage}}\medskip
}

\newcommand{\hackcenter}[1]{
 \xy (0,0)*{#1}; \endxy}

\def\bfx{\text{\boldmath$x$}}

\def\bfsig{\text{\boldmath$\sigma$}}

\def\bfempty{\text{\boldmath$\emptyset$}}

\def\bfp{\text{\boldmath$p$}}
\def\bfq{\text{\boldmath$q$}}
\def\bfr{\text{\boldmath$r$}}
\def\bfb{\text{\boldmath$b$}}
\def\bfy{\text{\boldmath$y$}}

\def\bfY{\text{\boldmath$Y$}}

\def\b1{\text{\boldmath$1$}}

\def\pmod#1{\text{ }(\text{\rm mod } #1)\,}

\newcommand{\Z}{\mathbb{Z}}

\def\phi{{\varphi}}

\newcommand{\la}{\langle}
\newcommand{\ra}{\rangle}

\newcommand{\IPGbl}{\operatorname{IP}_{G,b,\ell}}
\newcommand{\Tabqinf}{\operatorname{Tab}_{G,b,\ell}^{\bfq,\infty}}
\newcommand{\Tablq}{\operatorname{Tab}_{G,b,\ell}^{\preceq \mathbf{q}}}
\newcommand{\Tab}{\operatorname{Tab}_{G,b,\ell}}

\def\b{\mathfrak{b}}
\def\k{\Bbbk}

\theoremstyle{remark}

\def\height{{\operatorname{ht}}}

{\catcode`\|=\active
  \gdef\set#1{\mathinner{\lbrace\,{\mathcode`\|"8000%
  \let|\midvert #1}\,\rbrace}}
}
\def\midvert{\egroup\mid\bgroup}

\colorlet{darkgreen}{green!50!black}
\tikzset{dots/.style={very thick,loosely dotted},
         greendot/.style={fill,circle,color=darkgreen,inner sep=1.5pt,outer sep=0}
}
\def\greendot(#1,#2){\node[greendot] at(#1,#2){}}

\newenvironment{braid}{% sets defaults for the braid diagrams
  \begin{tikzpicture}[baseline=6mm,blue,line width=1pt, scale=0.4,
                      draw/.append style={rounded corners},
                      every node/.append style={font=\fontsize{5}{5}\selectfont}]%
  }{\end{tikzpicture}
}

\def\Grid(#1,#2){%  draws a coordinate grid inside a braid diagram
  \draw[very thin,gray,step=2mm] (0,0)grid(#1,#2);
  \draw[very thin,darkgreen,step=10mm] (0,0)grid(#1,#2);
}

\newcommand\Tableau[2][\relax]{
  \begin{tikzpicture}[scale=0.5,draw/.append style={thick,black}]
    \ifx\relax#1\relax%
    \else % shade the boxes in #1
      \foreach\box in {#1} { \filldraw[blue!30]\box+(-.5,-.5)rectangle++(.5,.5); }
    \fi
    \newcount\row\newcount\col
    \row=0
    \foreach \Row in {#2} {
       \col=1
       \foreach\k in \Row {
          \draw(\the\col,\the\row)+(-.5,-.5)rectangle++(.5,.5);
          \draw(\the\col,\the\row)node{\k};
          \global\advance\col by 1
       }
       \global\advance\row by -1
    }
  \end{tikzpicture}
}

\newcommand\YoungDiagram[2][\relax]{
  \begin{tikzpicture}[scale=0.5,draw/.append style={thick,black}]
    \ifx\relax#1\relax%
    \else % shade the boxes in #1
    \foreach\box in {#1} {
      \filldraw[blue!30]\box rectangle ++(1,1);
    }
    \fi
    \newcount\row
    \row=0
    \foreach \col in {#2} {
       \draw(1,\the\row)grid ++(\col,1);
       \global\advance\row by -1
    }
  \end{tikzpicture}
}

\begin{document}

%\Comment[AM]{I've added a \texttt{$\backslash$Comment\{\}} macro to help us write standout notes/comments/queries/etc to each other in the file. To mark them as your comments use something like \texttt{$\backslash$Comment[Sasha]\{\dots\}}  or  \texttt{$\backslash$Comment[Arun]\{\dots\}} etc.}

%%fakesection { title }
\title[The configuration space of a robotic arm over a graph]{{\bf The configuration space of a robotic arm over a graph}}

\author{\sc Derric Denniston}
\address{Washington \& Jefferson College\\ Washington\\ PA~15301, USA}
\email{dennistond@washjeff.edu}

\author{\sc Robert Muth}
\address{Department of Mathematics\\ Washington \& Jefferson College\\ Washington\\ PA~15301, USA}
\email{rmuth@washjeff.edu}

\author{\sc Vikram Singh}
\address{Washington \& Jefferson College\\ Washington\\ PA~15301, USA}
\email{singhvg@washjeff.edu}

%\subjclass[2010]{16G99}

%\thanks{Research supported by the NSF grant DMS-1161094 and the Humboldt Foundation.}

\begin{abstract}
We investigate the configuration space \(\mathcal{S}_{G,b,\ell}\) associated with the movement of a robotic arm of length \(\ell\) on a grid over an underlying graph \(G\), anchored at a vertex \(b \in G\). We study an associated PIP (poset with inconsistent pairs) \(\IPGbl\) consisting of indexed paths on \(G\). This PIP acts as a combinatorial model for the robotic arm, and we use \(\textup{IP}_{G,b,\ell}\) to show that the space \(\mathcal{S}_{G,b,\ell}\) is a CAT(0) cubical complex, generalizing work of Ardila, Bastidas, Ceballos, and Guo. This establishes that geodesics exist within the configuration space, and yields explicit algorithms for moving the robotic arm between different configurations in an optimal fashion. We also give a tight bound on the diameter of the robotic arm transition graph---the maximal number of moves necessary to change from one configuration to another---and compute this diameter for a large family of underlying graphs \(G\).
\end{abstract}

\maketitle

\section{Introduction}

In \cite{ABCG}, Ardila, Bastidas, Ceballos and Guo investigate the motion of a `robotic arm in a tunnel'. This robotic arm consists of a series of linked segments on a 2-dimensional \(m \times n\) grid capable of certain local movements. They use `coral tableaux' as a combinatorial model to study the configuration space of this robotic arm, and establish that this configuration space is a CAT(0) cubical complex. Our goal in the present paper is to prove similar results in the more general setting of `robotic arms over graphs'.

\begin{figure}[h]
\includegraphics[width=9cm]{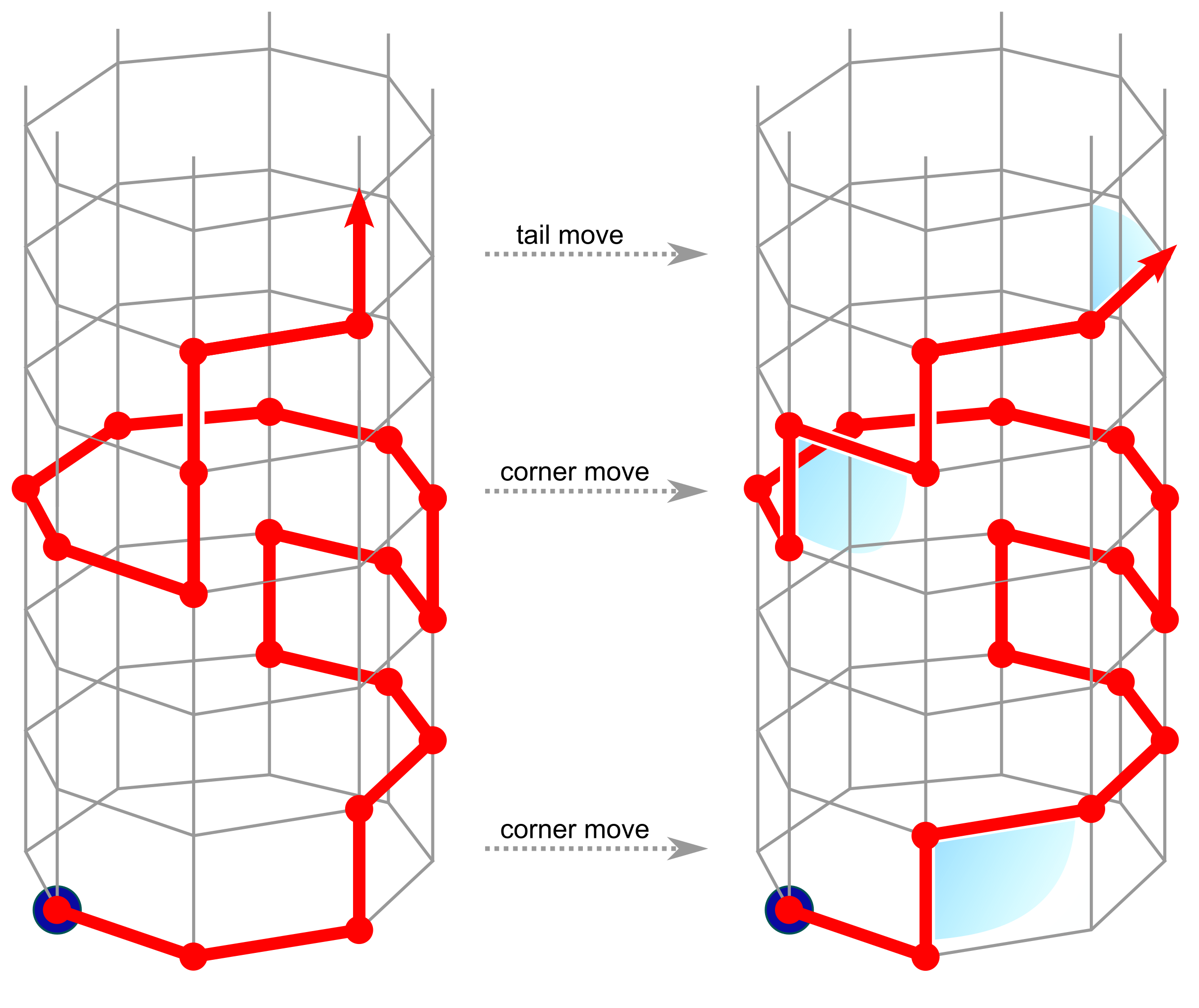}
\caption{On the left, a robotic arm configuration of length 20 over the cycle graph \(C_8\);  on the right, a new configuration achieved by applying three local moves to the configuration.}
\label{fig:configintro}       % Give a unique label
\end{figure}

\subsection{Robotic arms over graphs}\label{robint} For readability, we describe our setup somewhat informally in this introduction, and refer the reader to \S\ref{robotarms} for precise definitions. For a connected graph \(G\) with no loops or multiple edges, we consider \(G\) as the `ground floor' in the {\em workspace} \(W_G\) formed by extending an infinite rectangular grid above \(G\). In other words, \(W_G\) is an infinite stack of `floors' of \(G\), with vertical struts connecting the floors. The robotic arm \(\mathcal{R}_{G,b,\ell}\) is a sequence of \(\ell\) non-intersecting linked segments of nondecreasing height in \(W_G\), anchored at a vertex \(b\) in the ground floor \(G\). The robotic arm is capable of performing local moves which change its configuration, such as rotating its tail, and translating corners across squares in the grid \(W_G\). In Figure~\ref{fig:configintro}, we depict a configuration of the robotic arm of length \(\ell = 20\), over the cycle graph \(G = C_8\), along with a number of possible local moves available from this configuration. 

We note that this setup generalizes the `robotic arm in a tunnel', in the sense that one may view the robotic arms in \cite{ABCG} as the special case of \(G\) being the path graph \(A_m\) with \(m\) vertices, with anchor vertex \(b\) positioned at an endpoint. When \(G\) is a cycle graph, we may view \(\mathcal{R}_{G,b,\ell}\) as a `robotic arm on a cylinder', as in Figure~\ref{fig:configintro}.

\subsection{Configuration space}\label{configintro}
As in \cite{AG, ABY, AOS, Reeves, ABCG}, one may define a cubical complex \(\mathcal{S}_{G,b,\ell}\) associated with \(\mathcal{R}_{G,b,\ell}\); the 0-skeleton of this complex consists of the various configurations of the robotic arm, and the \(k\)-cubes correspond to \(k\)-sets of independent local moves for a given configuration---for instance, in Figure~\ref{fig:configintro}, note that the three highlighted local moves are independent, in that they may be performed simultaneously, or in any order, to arrive at the same configuration. See Figure~\ref{fig:fullspace} for a complete visualization of the configuration space of the length-5 robotic arm over the cycle graph \(C_3\).
Via the configuration space, one may translate natural questions about robotic arm movements into geometric questions about the space \(\mathcal{S}_{G,b,\ell}\), such as the following:\\

\vspace{-3mm}
\noindent{\bf Question.} For any two configurations, is there a process by which one may optimally move the robotic arm from one configuration to another:
\begin{enumerate}
\item[(I)] in a minimal number of total local moves, or; 
\item [(II)] in minimal  time, given that independent local moves can be performed simultaneously?
\end{enumerate}
As explained in \cite[\S5]{A}, questions (I) and (II) may be interpreted as a search for geodesics in \(\mathcal{S}_{G,b,\ell}\) under the \(L_1\)- and \(L_\infty\)-metrics, respectively. 

\subsection{The PIP of indexed paths}
In \S\ref{tabsec} we introduce a {\em PIP} (poset with inconsistent pairs) \(\IPGbl\), which consists of combinatorial objects called `indexed paths'. See Figure~\ref{fig:indpathfull} for the Hasse diagram of \(\IPGbl\) in the case of the length-5 robotic arm over the cycle graph \(C_3\). Following \cite{AOS, ABCG}, we define an associated cubical complex \(\mathcal{X}(\IPGbl)\) whose 0-skeleton consists of the consistent lower sets in \(\IPGbl\), and whose \(k\)-cubes correspond to \(k\)-sets of maximal elements within a given lower set. 

\subsection{Main results}\label{mainres}
Our first main result, which appears as Theorem~\ref{cubeisom} in the text, is as follows:

\begin{Theorem}\label{bigthm1}
There is an explicit isomorphism of cubical complexes \(\mathcal{X}(\textup{IP}_{G,b,\ell}) \cong \mathcal{S}_{G,b,\ell}\).
\end{Theorem}

As a byproduct, Theorem~\ref{bigthm1} establishes a number of useful results about \(\mathcal{S}_{G,b,\ell}\):
\begin{enumerate}
\item The cubical complex \(\mathcal{S}_{G,b,\ell}\) is a CAT(0) complex; it satisfies a certain non-positive curvature condition and is guaranteed to possess a unique geodesic between any two points (see Theorem~\ref{pipcat}).
\item We have positive and explicit answers to questions (I) and (II) above: Theorem~\ref{bigthm1}, in conjunction with the algorithms in \cite{ABY} can be used to optimally reconfigure the robotic arm (see Corollary~\ref{optimove}).
\end{enumerate}

\subsection{Diameter of the transition graph}
Another natural problem is to determine the maximum number of local moves needed to move the robotic arm \(\mathcal{R}_{G,b,\ell}\) from any configuration to any other. This value may be interpreted as the diameter of \(\mathcal{S}_{G,b,\ell}\) under the \(L_1\)-metric, or alternatively the diameter of the transition graph \(\mathcal{T}_{G,b,\ell}\)---the 1-skeleton of \(\mathcal{S}_{G,b,\ell}\). Our second main result, which appears as Theorem~\ref{diamthm} in the text, establishes a tight bound for this diameter:

\begin{Theorem}\label{bigthm2}
Let \(n\) be the number of vertices in \(G\). Then we have
\begin{align}\label{introineq}
\textup{diam}(\mathcal{T}_{G,b,\ell}) \leq 2 \left\lfloor \frac{(n - 1)(\ell + 1)^2}{2n}\right \rfloor,
\end{align}
where equality is achieved if there exist two cycle-free paths in \(G\) of length \(\min\{\ell, n-1\}\), originating at \(b\), with distinct initial edges.
\end{Theorem}

This bound is tight in that there exist large families of graphs \(G\) wherein (\ref{introineq}) is an equality, such as cycle graphs, complete graphs, and any graphs possessing a Hamiltonian circuit.

\subsection{Methods}
Our methods in this paper are entirely combinatorial, and we follow the approach of \cite{ABCG} in establishing Theorem~\ref{bigthm1}, which is an analogue of \cite[Theorem 5.5]{ABCG}. We introduce \(G\)-path tableaux as generalizations of the coral tableaux studied in \cite[\S5.2]{ABCG}, and rely on arguments invoking distributive lattice theory and Birkhoff's Theorem in the same manner as \cite[\S5.3]{ABCG}. The possible presence of non-trivial cycles in the underlying graph \(G\) forces some more delicate and technical definitions of combinatorial objects and their associated partial orders than in \cite{ABCG}, so we err on the side of providing complete proofs of necessary results in this more complicated setting rather than rely solely on analogy with \cite{ABCG}, though the spirit of the argument remains the same.

\subsection{Acknowledgements} 
Some of this work was completed in the summer of 2021, while the first and third authors were supported by an endowed fund administered by the Washington \& Jefferson College Mathematics Department.

\section{Preliminaries}

In this section we fix some notation and basic definitions on graphs and posets. For \(a, b\in \Z \) we write \([a,b] := \{c \in \Z \mid a \leq b \leq c\}\).

\subsection{Graphs and paths}\label{graphpath}
A {\em graph} (with no loops or multiple edges) \(H = (V_H, E_H)\) consists of a set of {\em vertices} \(V_H\) and a set of {\em edges} \(E_H \subseteq \{\{v,w\} \mid v,w \in V_H, v \neq w\}\). %We say edges \(e_1, e_2 \in E_H\) are {\em linked} if \(r_H(e_1) \cap r_H(e_2)\neq \varnothing\). 
An {\em \(H\)-path} \(\bfp\) is the data of a {\em length} \(\#\bfp \in \Z_{\geq 0}\), a sequence of edges \((\bfp_1, \ldots, \bfp_{\#\bfp}) \in E_H^{\#\bfp}\), and a sequence of vertices \((p_0, \ldots, p_{\#\bfp})\) such that \(\bfp_i = \{p_{i-1}, p_i\}\) for \(i \in [1, \#\bfp]\). Thus \(\bfp\) is a traversal of the vertices \(p_0, \ldots, p_{\#\bfp}\) in order by traveling along the edges \(\bfp_1, \ldots, \bfp_{\#\bfp}\) in order. 

For \(x \in V_H\), we say that \(\bfp\) is an {\em \((H,x)\)-path} if \(p_0 = x\). We note that the go-nowhere path \(\bfempty^x\) based at \(x\), defined by \(\#\bfempty^x = 0\) and \(\emptyset^x_0 = x\) is an \((H,x)\)-path. We say that an \(H\)-path (resp. \((H,x)\)-path) is an \(H^+\)-path (resp. \((H,x)^+\)-path) if \(\#\bfp > 0\).

 If \(\bfp, \bfq\) are paths such that \(p_{\#\bfp} = q_0\), then we write \(\bfp \bfq\) for their concatenation, which has \(\#(\bfp \bfq) = \#\bfp + \#\bfq\), edge sequence \((\bfp_1, \ldots, \bfp_{\#\bfp}, \bfq_1, \ldots,, \bfq_{\#\bfq})\), and vertex sequence \((p_0, \ldots, p_{\#\bfp}, q_1, \ldots, q_{\#\bfq})\).
If \(\bfr = \bfp \bfq\), we say that \(\bfp\) is a prefix of \(\bfr\) and \(\bfq\) is a suffix of \(\bfr\). We also write in this case \(\bfp \preceq^\textup{pre} \bfr\), and note that \(\preceq^\textup{pre}\) defines a partial order on the set of all \(H\)-paths.

We say an \(H\)-path \(\bfp\) is {\em cycle-free} if \(p_0, \ldots, p_{\#\bfp}\) are distinct. For an \(H^+\)-path \(\bfp\), the {\em maximal-length cycle-free suffix decomposition of \(\bfp\)} is the data of a unique integer \(n_\bfp\) and unique set of \(n_{\bfp}\) \(H^+\)-paths \(\bfp^{(1)}, \ldots, \bfp^{(n_\bfp)}\) such that \(\bfp = \bfp^{(1)} \cdots \bfp^{(n_\bfp)}\) and \(\bfp^{(t)}\) is the maximal-length cycle-free suffix of \(\bfp^{(1)} \cdots \bfp^{(t)}\) for \(t \in[1, n_\bfp]\). For \(r \in [1,\#\bfp]\), we set \(d_{\bfp}(r)\) to be the unique integer \(t\in[1, n_\bfp]\) such that 
\begin{align*}
 r \in [\#\bfp^{(1)} + \cdots + \#\bfp^{(t -1)} + 1, \#\bfp^{(1)} + \cdots + \#\bfp^{(t )}],
\end{align*}
so that informally speaking, the \(r\)th edge \(\bfp_r\) in \(\bfp\) is an edge in \(\bfp^{(t)}\) when \(d_\bfp(r) = t\). See Figure~\ref{fig:pathdecex} for a visual depiction.

\begin{figure}[h]
\begin{align*}
\hackcenter{
\begin{tikzpicture}[scale=0.6]
\draw[black, fill =black]  (0,0) circle (3pt);
\draw[black, fill =black]  (3,1) circle (3pt);
\draw[black, fill =black]  (-1,2) circle (3pt);
\draw[black, fill =black]  (5,1) circle (3pt);
\draw[ thick, join=round, cap=round] (0,0)--(3,1)--(-1,2)--(0,0);
\draw[ thick, join=round, cap=round] (3,1)--(5,1);
  \node[above] at (0.2,0){ $\scriptstyle b$};    
   \node[above] at (3,1.3){ $\scriptstyle a$};    
    \node[below] at (-0.6,2){ $\scriptstyle c$};    
     \node[right] at (5,1){ $\scriptstyle d$};    
%%%
%%%
 \draw[ ultra thick, join=round, cap=round, red] (0,0) .. controls ++(.45,-0.2) and ++(-.25,-0.35) ..(3,1);
  \draw[ ultra thick, join=round, cap=round, red] (3,1) .. controls ++(.45,-0.45) and ++(-.45,-0.45) .. (5,1);
    \draw[ ultra thick, join=round, cap=round, red] (5,1) .. controls ++(-.45,0.45) and ++(+.45,0.45) .. (3,1.2);
       \draw[ ultra thick, join=round, cap=round, red] (3,1.2) .. controls ++(-.45,0.45) and ++(+.45,0.25) .. (-1,2);
        \draw[ ultra thick, join=round, cap=round, red] (-1,2) .. controls ++(-.45,-0.45) and ++(-.45,0.25) .. (-0.2,-0.2);
          \draw[ ultra thick, join=round, cap=round, red, ->] (-0.2,-0.2) .. controls ++(.45,-0.45) and ++(-.55,-0.5) .. (3,0.7);
\end{tikzpicture}
}
\;\;
\leadsto
\;\;
\hackcenter{
\begin{tikzpicture}[scale=0.5]
\draw[black, fill =black]  (0,0) circle (3pt);
\draw[black, fill =black]  (3,1) circle (3pt);
\draw[black, fill =black]  (-1,2) circle (3pt);
\draw[black, fill =black]  (5,1) circle (3pt);
\draw[ thick, join=round, cap=round] (0,0)--(3,1)--(-1,2)--(0,0);
\draw[ thick, join=round, cap=round] (3,1)--(5,1);
  \node[above] at (0.2,0){ $\scriptstyle b$};    
   \node[above] at (3,1.1){ $\scriptstyle a$};    
    \node[below] at (-0.55,2){ $\scriptstyle c$};    
     \node[right] at (5,1){ $\scriptstyle d$};    
%%%
%%%
 \draw[ ultra thick, join=round, cap=round, blue] (0,0) .. controls ++(.55,-0.2) and ++(-.25,-0.45) ..(3,1);
  \draw[ ultra thick, join=round, cap=round, blue, ->] (3,1) .. controls ++(.45,-0.45) and ++(-.45,-0.45) .. (4.9,0.9);
\end{tikzpicture}
}
\;\;
\hackcenter{
\begin{tikzpicture}[scale=0.5]
\draw[black, fill =black]  (0,0) circle (3pt);
\draw[black, fill =black]  (3,1) circle (3pt);
\draw[black, fill =black]  (-1,2) circle (3pt);
\draw[black, fill =black]  (5,1) circle (3pt);
\draw[ thick, join=round, cap=round] (0,0)--(3,1)--(-1,2)--(0,0);
\draw[ thick, join=round, cap=round] (3,1)--(5,1);
  \node[above] at (0.2,0){ $\scriptstyle b$};    
   \node[above] at (3,1.1){ $\scriptstyle a$};    
    \node[below] at (-0.55,2){ $\scriptstyle c$};    
     \node[right] at (5,1){ $\scriptstyle d$};    
%%%
%%%
    \draw[ ultra thick, join=round, cap=round, blue] (5,1) .. controls ++(-.45,0.45) and ++(+.45,0.45) .. (3,1);
       \draw[ ultra thick, join=round, cap=round, blue, ->] (3,1) .. controls ++(-.45,0.45) and ++(+.65,0.25) .. (-1,2.2);
\end{tikzpicture}
}
\hackcenter{
\begin{tikzpicture}[scale=0.5]
\draw[black, fill =black]  (0,0) circle (3pt);
\draw[black, fill =black]  (3,1) circle (3pt);
\draw[black, fill =black]  (-1,2) circle (3pt);
\draw[black, fill =black]  (5,1) circle (3pt);
\draw[ thick, join=round, cap=round] (0,0)--(3,1)--(-1,2)--(0,0);
\draw[ thick, join=round, cap=round] (3,1)--(5,1);
  \node[above] at (0.2,0){ $\scriptstyle b$};    
   \node[above] at (3,1.1){ $\scriptstyle a$};    
    \node[below] at (-0.55,2){ $\scriptstyle c$};    
     \node[right] at (5,1){ $\scriptstyle d$};    
%%%
%%%
        \draw[ ultra thick, join=round, cap=round, blue] (-1,2) .. controls ++(-.45,-0.45) and ++(-.45,0.25) .. (0,0);
          \draw[ ultra thick, join=round, cap=round, blue, ->, shorten >=0.15cm] (0,0) .. controls ++(.45,-0.45) and ++(-.55,-0.5) .. (3,1);
\end{tikzpicture}
}
\end{align*}
\caption{A path \(\bfp\) of length \(6\); the maximal length cycle-free suffix decomposition \(\bfp^{(1)} \bfp^{(2)} \bfp^{(3)}\) of \(\bfp\), where \(n_\bfp = 3\). We have \(d_\bfp(1) = d_\bfp(2) = 1\), \(d_\bfp(3) = d_\bfp(4) = 2\), and \(d_\bfp(5) = d_\bfp(6) = 3\). }
\label{fig:pathdecex}
\end{figure}
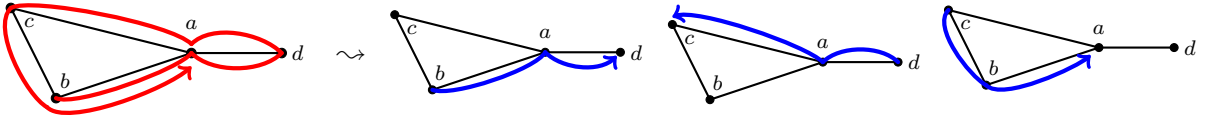

\subsection{Posets} 
A {\em partially ordered set} (or {\em poset}) is a set \(\lambda\) together with a binary relation \(\succeq\), which satisfies the following conditions for all \(u,v,w \in \lambda\):
\begin{enumerate}
\item  \(u \succeq u\) ({\em reflexivity});
\item \(u \succeq v\) and \(v\succeq u\) imply \(u = v\) ({\em antisymmetricity});
\item \(u \succeq v\) and \(v \succeq w\) imply \(u \succeq w\) ({\em transitivity}).
\end{enumerate}
We use \(a \succ b\) to indicate \(a \succeq b\) and \(a \neq b\). An {\em order-preserving} map of posets \(\lambda, \nu\) is a set map \(f: \lambda \to \nu\) such that \(f(u) \succeq f(v)\) whenever \(u \succeq v\). We say two posets \(\lambda, \nu\) are {\em isomorphic} and write \(\lambda \cong \nu\) if there exist mutually inverse order-preserving maps \(\lambda \rightleftarrows \nu\).

We say in \(\mu \subseteq \lambda\) is a {\em lower set} if \(u \in \mu\) whenever \(u \preceq v\) for some \(v \in \mu\). For \(\nu \subseteq \lambda\), we will write 
\begin{align*}
\mathcal{I}(\nu) := \{ u \in \lambda \mid u \preceq v \textup{ for some } v \in \nu\}
\end{align*}
for the lower set generated by \(\nu\).

\subsection{Distributive lattices} 
A {\em lattice} is a poset \(\lambda\) in which every pair of elements \(u,v \in \lambda\) has a unique supremum \(u \vee v\) called the {\em join}, and a unique infimum \(u \wedge v\) called the {\em meet}. We say that a lattice is {\em distributive} if, in addition, we have
\begin{align*}
u \wedge (v \vee w) = (u \wedge v) \vee (u \wedge w).
\end{align*}
We say that an element \(u\) in a distributive lattice \(\lambda\) is {\em join-irreducible} if \(u = v \vee w\) implies \(u \in \{v,w\}\). We will write \(J(\lambda)\) for the poset (under the induced partial order) of join-irreducible elements in \(\lambda\).

For a poset \(\mu\), we write \(\mathcal{L}(\mu)\) for the {\em distributive lattice of lower sets in \(\mu\)}. The elements of \(\mathcal{L}(\mu)\) are the lower sets in \(\mu\), the partial order in \(\mathcal{L}(\mu)\) is given by set inclusion, and join and meet are thus given by union and intersection respectively. We have the following fundamental result:

\begin{Theorem}[Birkhoff's Theorem \cite{Birk}]\label{Birk}
If \(\lambda\) is a finite distributive lattice then the map
\begin{align*}
B: \lambda \to \mathcal{L}(J(\lambda)), \qquad x \mapsto \{j \in J(\lambda) \mid j \preceq x\}
\end{align*}
is an isomorphism of distributive lattices.
\end{Theorem}

\subsection{PIPs}
Following \cite{AOS, Winskel}, we now define PIPs and their associated cubical complexes.
A {\em poset with inconsistent pairs}, or {\em PIP}, is a poset \(\lambda\) together with an additional symmetric `inconsistency' relation \(\nleftrightarrow\) on \(\lambda\) satisfying the condition:
\begin{align*}
u \nleftrightarrow v \preceq w \qquad \implies\qquad u \nleftrightarrow w.
\end{align*}
We say that a lower set \(\mu \subseteq \lambda\) is {\em consistent} if there are no \(u,v \in \mu\) such that \(u \nleftrightarrow v\). Given a PIP \(\mu\), we write \(\mathcal{L}_{\textup{con}}(\mu)\) for the set of {\em consistent} lower sets in \(\mu\), with partial order given by inclusion.

\subsection{Cubical complexes}
A {\em cubical complex} is a polyhedral complex where all cells are \(k\)-cubes and all attaching
maps are injective. We may describe cubical complexes by defining a 0-skeleton of vertices, then inductively attaching \(k\)-cubes by noting the \(2^k\) vertices in the 0-skeleton that form the vertices of the \(k\)-cube, and the \(2k\) previously attached \((k-1)\)-cubes which serve as attaching faces for the \(k\)-cube. A {\em rooted} cubical complex has a designated root vertex. See \cite{Pratt, GP} for a more detailed discussion of the geometry of cubical complexes and their connection to automata, scheduling, and reconfiguration study.

In this paper we are particularly interested in cubical complexes which are {\em CAT(0)} metric spaces. A metric space \(X\) is said to be CAT(0) provided that there is a unique geodesic path in \(X\) between any two points, and \(X\) has non-positive global curvature. See \cite{AG, ABY, AOS, Reeves} for a complete discussion. Our main interest lies in the fact that when the configuration space of a robotic arm may be viewed as a CAT(0) cubical complex, there exists an explicit combinatorial algorithm (see \cite{ABY}) for optimally reconfiguring the robot.

\subsection{The rooted cubical complex associated to a PIP}
Let \(\lambda\) be a PIP. We now define a rooted cubical complex \(\mathcal{X}(\lambda)\) associated to \(\lambda\). The 0-skeleton of \(\mathcal{X}(\lambda)\) is the set \(\mathcal{L}_{\textup{con}}(\lambda)\). Cubes are added as follows. Let \(\mu \in \mathcal{L}_{\textup{con}}(\lambda)\). Let \(K \subseteq \mu\) be a subset of maximal elements in \(\mu\); i.e., \(u \in K\) implies \(u \preceq v\) for all \(v \in \mu\). Then the consistent lower sets \(\{ \nu \mid \mu \backslash K \subseteq \nu \subseteq \mu\}\) form the vertices of a \(|K|\)-cube \([\mu; K]\)  in \(\mathcal{X}(\lambda)\). 
The boundary of this cube is the collection of \(2|K|\) faces 
\(
\bigcup_{u \in K} 
[ \mu ; K \backslash \{u\}]
 \cup
[\mu \backslash \{u\}; K \backslash \{u\}].
\)
The empty set \(\varnothing\) serves as the root in the cubical complex \(\mathcal{X}(\lambda)\). See \cite[\S4.2]{ABCG} for a visual depiction of a PIP with associated cubical complex.

\subsection{PIPs and CAT(0) cubical complexes} As discussed, for instance in \cite{ABY, ABCG, Roller, Sageev, AOS}, PIPs prove to be a useful combinatorial framework for identifying and describing CAT(0) complexes, thanks to the following theorem:

\begin{Theorem}[\cite{Roller, Sageev, AOS}]\label{pipcat}
The map \(\lambda \mapsto \mathcal{X}(\lambda)\) is a bijection of PIPs and rooted CAT(0) cubical complexes.
\end{Theorem}

Thus, to prove that a (rooted) cubical complex is CAT(0), it suffices to demonstrate that it is isomorphic to \(\mathcal{X}(\lambda)\) for some PIP \(\lambda\). Following \cite{ABY}, one may use \(\lambda\) combinatorics as a `remote control' to algorithmically generate geodesics between vertices in \(\mathcal{X}(\lambda)\).

\section{Indexed paths and tableaux}\label{tabsec}
From this point forward, we fix a connected graph \(G\) with no loops or multiple edges, a vertex \(b \in V_G\), and a nonnegative integer \(\ell \in \Z_{\geq 0}\). In this section we define and study a number of combinatorial objects associated with the data \((G,b,\ell)\). The results in this section will establish a PIP which serves as a combinatorial model, or `remote control' for the workings of the robotic arm \(\mathcal{R}_{G,b,\ell}\) described in \S\ref{robint}.

\subsection{Indexed paths}
An {\em indexed path} will be a symbol of the form \(\la \bfp, a \ra\), where we define
\begin{align}\label{indpathspec}
\textup{IP}_{G,b,\ell} := 
\{\la\bfp, a\ra \mid \bfp \textup{ is a \((G,b)^+\)-path}, \; a \in [0, \ell +1 - \#\bfp - n_{\bfp}]\},
\end{align}
and refer to the elements of \(\IPGbl\) as {\em indexed \((G,b,\ell)\)-paths}. We will visually depict an indexed \((G,b,\ell)\)-path \(\langle \bfp, a \rangle\) as a path \(\bfp\) accompanied by a circled integer $\raisebox{.5pt}{\textcircled{\raisebox{-.9pt} {\(a\)}}}$.

\subsubsection{The PIP of indexed paths}\label{pipind}
We define relations on \(\IPGbl\) as follows. For \(\la\bfp, a\ra, \la\bfq, b\ra \in \IPGbl\), write \(\la\bfp,a\ra \preceq^\textup{IP} \la\bfq,b\ra\) provided that:
\begin{align*}
\textup{(i)}\;
\bfp \preceq^\textup{pre} \bfq
\qquad
\textup{and}
\qquad
\textup{(ii)}\;
n_\bfp + a \geq d_\bfq(\#\bfp) + b.
\end{align*}
If neither of \(\bfp, \bfq\) is a prefix of the other, then write \(\la\bfp,a\ra \nleftrightarrow^\textup{IP} \la\bfq,b\ra\). In Lemma~\ref{isPIP} we will establish that this defines a PIP structure on \(\IPGbl\).

\begin{Example}
In Figure~\ref{fig:indpathfull}, the Hasse diagram of the PIP of indexed paths \(\textup{IP}_{C_3, b, 5}\) is depicted, where \(C_3\) is the cycle graph on 3 vertices. Incosistent pairs are indicated by connecting each \(\preceq\)-minimal inconsistent pair with a dotted line.
\end{Example}

\begin{landscape}
\begin{figure}
{}
\vspace{2cm}
{}
\includegraphics[width=22.5cm]{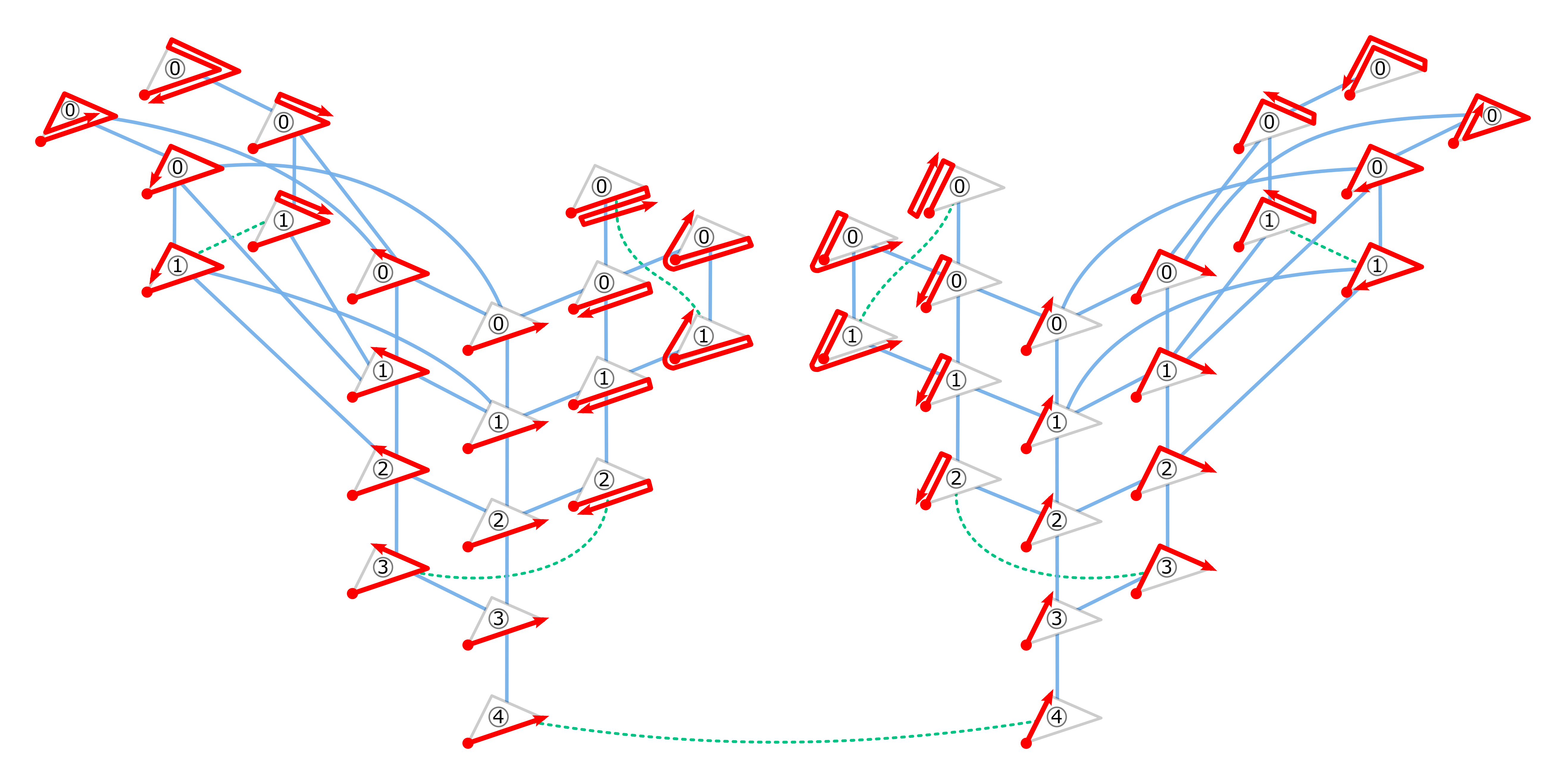}
\caption{Hasse diagram for the PIP of indexed paths \(\textup{IP}_{C_3, b, 5}\).}
\label{fig:indpathfull}       % Give a unique label
\end{figure}
\end{landscape}

\begin{Lemma}\label{isPIP}
The relations \(\preceq^\textup{IP}\), \(\nleftrightarrow^\textup{IP}\) define a PIP structure on \(\IPGbl\).
\end{Lemma}
\begin{proof}

Reflexivity is obvious, since \(n_\bfp =  d_\bfp(\#\bfp)\) by definition. For antisymmetry, assume  \(\la\bfp, a\ra \preceq^\textup{IP} \la\bfq, b\ra\) and \(\la\bfq, b\ra\preceq^\textup{IP} \la\bfp, a\ra\). Then \(\bfp, \bfq\) are mutual prefixes, so we must have \(\bfp = \bfq\). Then we have \(d_\bfp(\#\bfq) = d_\bfq(\#\bfp) = d_\bfp(\#\bfp) = n_\bfp\), so condition (ii) forces \(a \geq b\) and \(b \geq a\), so \(a=b\), and thus \(\la\bfp,a\ra = \la\bfq, b\ra\), as desired.

To prove transitivity, assume \(\la\bfp, a\ra \preceq^\textup{IP} \la\bfq, b\ra \preceq^\textup{IP} \la\bfr, c\ra\). Then we have \(\bfp \preceq^\textup{pre} \bfq \preceq^\textup{pre} \bfr\), and
\begin{align*}
n_\bfp + a \geq d_\bfq(\#\bfp) + b \geq d_\bfq(\#\bfp) + d_{\bfr}(\#\bfq) + c - n_\bfq \geq d_{\bfr}(\#\bfq) + c \geq d_\bfr(\#\bfp) + c,
\end{align*}
where the first inequality comes from the fact that \(\la\bfp,a\ra \preceq^\textup{IP} \la\bfq,b\ra\), the second inequality comes from the fact that \(\la\bfq,b\ra \preceq^\textup{IP} \la\bfr, c\ra\), the third inequality comes from the fact that \(\#\bfp \leq \#\bfq\) and so \(n_\bfq \geq d_\bfq(\#\bfp)\) by definition, and the fourth inequality comes from the fact that the function \(d_\bfr\) is weakly increasing. Thus \(\la\bfp, a\ra \preceq^\textup{IP} \la\bfr, c\ra\).

To prove the inconsistency axiom, assume \(\la\bfp, a\ra \nleftrightarrow^\textup{IP} \la\bfq, b\ra \preceq^\textup{IP} \la\bfr, c\ra\). Then \(\bfq \preceq^{\textup{pre}}\bfr\). If \(\bfp \preceq^{\textup{pre}}\bfr\), then it would follow that either \(\bfp \preceq^{\textup{pre}}\bfq\) or \(\bfq \preceq^{\textup{pre}}\bfp\), a contradiction. If, on the other hand, \(\bfr \preceq^{\textup{pre}} \bfp\), then we would have that \(\bfq \preceq^{\textup{pre}} \bfr  \preceq^{\textup{pre}} \bfp\), another contradiction. Thus \(\la\bfp, a\ra \nleftrightarrow^\textup{IP} \la\bfr, c\ra\) as desired.
\end{proof}

\subsection{\((G,b,\ell)\)-tableaux} 
\begin{Definition}\label{deftab}
A {\em \((G,b,\ell)\)-tableau} \((\bfp,L)\) is the data of a \((G,b)\)-path \(\bfp\) and a `labeling' function \(L:[1, \#\bfp] \to \Z_{\geq 0}\) such that:
\begin{enumerate}
\item Labels are weakly increasing: \(L(i) \leq L(j)\) for \(i \leq j\).
\item If \(i<j\) and \(p_{i-1} = p_j\), then \(L(i) <L(j)\).
\item \(L(\#\bfp) + \#\bfp \leq \ell\).
\end{enumerate}
\end{Definition}
We note that (ii) is equivalent to asserting that the \(G\)-path \((\bfp_t)_{t \in L^{-1}(m)}\) is cycle-free for \(m \in \Z_{\geq 0}\).
We also note that \((\bfempty^b, \varnothing \to \Z_{\geq 0})\), the go-nowhere path at \(b\) with trivial labeling function  is a \((G,b,\ell)\)-tableau. We write \(\textup{Tab}_{G,b,\ell}\) for the set of all \((G,b,\ell)\)-tableaux. We will visually depict \((G,b,\ell)\)-tableaux as labeled paths wherein the \(i\)th edge is labeled by \(L(i)\).

\subsection{Tight \((G,b,\ell)\)-tableaux}
For \(\la \bfp, a \ra \in \IPGbl\), define a labeling function \(L_{\la \bfp, a\ra} : [1, \#\bfp] \to \Z_{\geq 0}\) by setting
\(
L_{\la \bfp, a\ra}(r) = d_\bfp(r) + a - 1.
\)

\begin{Lemma}
The assignment \(\la \bfp, a \ra \mapsto (\bfp, L_{\la \bfp, a\ra})\) gives a well-defined function \(\tau:\IPGbl \to \Tab\).
\end{Lemma}
\begin{proof}
We check that \((\bfp, L_{\la \bfp, a\ra})\) satisfies axioms (i)--(iii) of Definition~\ref{deftab}. Part (i) follows from the fact that \(d_\bfp\) is weakly increasing by definition. For (ii) we note that if \(i<j\), \(p_{i-1} = p_j\), then \(\bfp_i, \bfp_j\) cannot belong to the same part of the the maximal-length cycle-free suffix decomposition \(\bfp^{(1)}, \ldots, \bfp^{(n_\bfp)}\) of \(\bfp\). Thus \(d_\bfp(i) < d_\bfp(j)\), so \(L_{\la \bfp, a\ra}(i) < L_{\la \bfp, a\ra}(j)\). For (iii), we have
\begin{align*}
L_{\la \bfp, a\ra}(\#\bfp) + \#\bfp = d_\bfp(\#\bfp) + a -1 + \#\bfp = n_\bfp + a -1 + \#\bfp \leq \ell,
\end{align*}
where the last inequality follows from (\ref{indpathspec})
and the fact that \(\la \bfp, a \ra \in \IPGbl\).
\end{proof}

We refer to members of \(\tau(\IPGbl) \subseteq \Tab\) as {\em tight \((G,b,\ell)\)-tableaux}. See Figure~\ref{fig:tighttabex} for a visual depiction of tight and non-tight tableaux.

\begin{figure}[h]
\begin{align*}
\hackcenter{
\begin{tikzpicture}[scale=0.6]
\draw[black, fill =black]  (0,0) circle (3pt);
\draw[black, fill =black]  (3,1) circle (3pt);
\draw[black, fill =black]  (-1,2) circle (3pt);
\draw[black, fill =black]  (5,1) circle (3pt);
\draw[ thick, join=round, cap=round] (0,0)--(3,1)--(-1,2)--(0,0);
\draw[ thick, join=round, cap=round] (3,1)--(5,1);
  \node[above] at (0.5,0.5){ $\raisebox{.5pt}{\textcircled{\raisebox{-.9pt} {2}}}$};     
%%%
%%%
 \draw[ ultra thick, join=round, cap=round, red] (0,0) .. controls ++(.25,0.5) and ++(-.45,0.15) ..(3,1);
  \draw[ ultra thick, join=round, cap=round, red] (3,1) .. controls ++(.45,-0.45) and ++(-.45,-0.45) .. (5,1);
    \draw[ ultra thick, join=round, cap=round, red] (5,1) .. controls ++(-.45,0.45) and ++(+.45,0.45) .. (3,1.2);
       \draw[ ultra thick, join=round, cap=round, red] (3,1.2) .. controls ++(-.45,0.45) and ++(+.45,0.25) .. (-1,2);
        \draw[ ultra thick, join=round, cap=round, red] (-1,2) .. controls ++(-.45,-0.45) and ++(-.45,0.25) .. (-0.2,-0.2);
          \draw[ ultra thick, join=round, cap=round, red, ->] (-0.2,-0.2) .. controls ++(.45,-0.45) and ++(-.55,-0.5) .. (3,0.7);
\end{tikzpicture}
}
\;\;
\xrightarrow{\tau}
\;\;
\hackcenter{
\begin{tikzpicture}[scale=0.6]
\draw[black, fill =black]  (0,0) circle (3pt);
\draw[black, fill =black]  (3,1) circle (3pt);
\draw[black, fill =black]  (-1,2) circle (3pt);
\draw[black, fill =black]  (5,1) circle (3pt);
\draw[ thick, join=round, cap=round] (0,0)--(3,1)--(-1,2)--(0,0);
\draw[ thick, join=round, cap=round] (3,1)--(5,1);
  \node[above] at (0.5,0.5){ $\scriptstyle 2$};     
    \node[above] at (4,-0.1){ $\scriptstyle 2$};    
     \node[above] at (4,1.4){ $\scriptstyle 3$};   
     \node[above] at (1,1.8){ $\scriptstyle 3$};   
       \node[above] at (-1.4,0.6){ $\scriptstyle 4$};  
             \node[above] at (1,-1){ $\scriptstyle 4$};  
%%%
%%%
 \draw[ ultra thick, join=round, cap=round, red] (0,0) .. controls ++(.25,0.5) and ++(-.45,0.15) ..(3,1);
  \draw[ ultra thick, join=round, cap=round, red] (3,1) .. controls ++(.45,-0.45) and ++(-.45,-0.45) .. (5,1);
    \draw[ ultra thick, join=round, cap=round, red] (5,1) .. controls ++(-.45,0.45) and ++(+.45,0.45) .. (3,1.2);
       \draw[ ultra thick, join=round, cap=round, red] (3,1.2) .. controls ++(-.45,0.45) and ++(+.45,0.25) .. (-1,2);
        \draw[ ultra thick, join=round, cap=round, red] (-1,2) .. controls ++(-.45,-0.45) and ++(-.45,0.25) .. (-0.2,-0.2);
          \draw[ ultra thick, join=round, cap=round, red, ->] (-0.2,-0.2) .. controls ++(.45,-0.45) and ++(-.55,-0.5) .. (3,0.7);
\end{tikzpicture}
}
\qquad
\;\;\;
\hackcenter{
\begin{tikzpicture}[scale=0.6]
\draw[black, fill =black]  (0,0) circle (3pt);
\draw[black, fill =black]  (3,1) circle (3pt);
\draw[black, fill =black]  (-1,2) circle (3pt);
\draw[black, fill =black]  (5,1) circle (3pt);
\draw[ thick, join=round, cap=round] (0,0)--(3,1)--(-1,2)--(0,0);
\draw[ thick, join=round, cap=round] (3,1)--(5,1);
  \node[above] at (0.5,0.5){ $\scriptstyle 2$};     
    \node[above] at (4,-0.1){ $\scriptstyle 2$};    
     \node[above] at (4,1.4){ $\scriptstyle 3$};   
     \node[above] at (1,1.8){ $\scriptstyle 3$};   
       \node[above] at (-1.4,0.6){ $\scriptstyle 3$};  
             \node[above] at (1,-1){ $\scriptstyle 4$};  
%%%
%%%
 \draw[ ultra thick, join=round, cap=round, blue] (0,0) .. controls ++(.25,0.5) and ++(-.45,0.15) ..(3,1);
  \draw[ ultra thick, join=round, cap=round, blue] (3,1) .. controls ++(.45,-0.45) and ++(-.45,-0.45) .. (5,1);
    \draw[ ultra thick, join=round, cap=round, blue] (5,1) .. controls ++(-.45,0.45) and ++(+.45,0.45) .. (3,1.2);
       \draw[ ultra thick, join=round, cap=round, blue] (3,1.2) .. controls ++(-.45,0.45) and ++(+.45,0.25) .. (-1,2);
        \draw[ ultra thick, join=round, cap=round, blue] (-1,2) .. controls ++(-.45,-0.45) and ++(-.45,0.25) .. (-0.2,-0.2);
          \draw[ ultra thick, join=round, cap=round, blue, ->] (-0.2,-0.2) .. controls ++(.45,-0.45) and ++(-.55,-0.5) .. (3,0.7);
\end{tikzpicture}
}
\end{align*}
\caption{An indexed path \(\langle \bfp, 2\rangle\) and associated tight tableau \(\tau(\langle \bfp, 2 \rangle)\); a non-tight tableau on the same path \(\bfp\).}
\label{fig:tighttabex}
\end{figure}
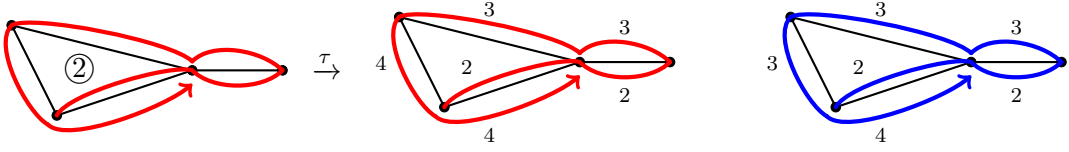

\begin{Lemma}\label{tightimage}
Let \((\bfp, L) \in \Tab\), and set \(a = L(\#\bfp) - n_\bfp + 1\). Then \((\bfp, L) = \tau(\la \bfp, a \ra)\) if and only if, for every \((\bfp, L') \in \Tab\) with \(L'(\#\bfp) = L(\#\bfp)\), we have \(L'(r) \leq L(r)\) for all \(r \in [1, \#\bfp]\).
\end{Lemma}
\begin{proof}
Let \((\bfp, L') \in \Tab\). Define the sequence \(u_1, \ldots, u_{n_\bfp}\) by letting \(u_i\) be maximal such that \(d_{\bfp}(u_i) = i\). Then by definition of the maximal-length cycle-free suffix decomposition of \(\bfp\), for \(i \in [1, n_\bfp -1]\) there exists \(v_i\) such that \(u_i < v_i \leq u_{i+1}\) and \(p_{u_i-1} = p_{v_i}\). Then by Definition~\ref{deftab}(ii), we must have
\(
L'(u_1) < L'(v_1) \leq L'(u_{i+1})
\)
for all \(i \in [1,n_{\bfp} -1]\), so \(L'(u_j) \geq j-1\) and \(L'(\#\bfp) = L'(u_{n_\bfp}) \geq L'(u_j) + n_{\bfp} -j \) for all \(j\). 

Now let \((\bfp, L) \in \Tab\), and set \(a = L(\#\bfp) - n_\bfp + 1\). By the above paragraph, \(n_{\bfp} -1 \leq L(\#\bfp) \leq \ell - \#\bfp\), so \(a \in [0, \ell +1 - \#\bfp - n_{\bfp}]\). Thus, in consideration of (\ref{indpathspec}), there exists an element \(\la \bfp, a \ra \in \IPGbl\). Let \((\bfp, L_{\la \bfp, a\ra}) = \tau(\la \bfp, a \ra)\) be the associated tight \((G,b,\ell)\)-tableau. Note that \(L_{\la \bfp, a\ra}(\#\bfp) = n_\bfp + a -1 = L(\#\bfp)\).

\((\implies)\)
Assume by way of contradiction that \((\bfp, L') \in \Tab\) is such that \(L'(\#\bfp) = L_{\la \bfp, a\ra}(\#\bfp)\), and 
\(L'(r) > L_{\la \bfp, a\ra}(r)\) for some \(r \in [1, \#\bfp]\). We have \(d_{\bfp}(r) = d_{\bfp}(u_j)\) for some \(j\), and \(u_j \geq r\). Hence by Definition~\ref{deftab}(i), \(L'(u_j) \geq L'(r) > L_{\la \bfp, a\ra}(r) =L_{\la \bfp, a\ra}(u_j) = d_\bfp(u_j) + a - 1 = j+ a - 1\). Therefore \(L'(\#\bfp) > n_\bfp + a-1 = L_{\la \bfp, a\ra}(\#\bfp)\), giving the desired contradiction. 

\((\impliedby)\) By the above, we have \(L(r) \leq L_{\la \bfp, a\ra}(r)\) for all \(r \in [1,\#\bfp]\). On the other hand, if \(L(r) \geq L_{\la \bfp, a\ra}(r)\) for all \(r \in [1,\#\bfp]\), it then follows that \(L(r) = L_{\la \bfp, a\ra}(r)\) for all \(r \in [1,\#\bfp]\), and hence \((\bfp,L) = (\bfp, L_{\la \bfp, a\ra})\), completing the proof.
\end{proof}

\subsection{Extended tableaux}\label{exttab}

Let \(\bfq\) be a \((G,b)\)-path. We write \(\textup{Tab}_{G,b,\ell}^{\preceq\bfq}\) for the set of \((G,b,\ell)\)-tableaux \((\bfp, L)\) such that \(\bfp \preceq^\textup{pre} \bfq\). 
An {\em extended \(\bfq\)-tableau} is a weakly increasing labeling function \(U:[1, \#\bfq] \to \Z_{\geq 0} \cup \{\infty\}\) such that, if \(U(m) \in \Z_{\geq 0}\), then \(((\bfq_1, \ldots, \bfq_m), U|_{[1,m]})\) is a \((G,b,\ell)\)-tableau. Writing \(\Tabqinf\) for the set of extended \(\bfq\)-tableaux, there is a clear bijection \(g^\bfq: \Tablq \to \Tabqinf\) given by sending \((\bfp, L)\) to \(L^{\bfq, \infty}\), where
\begin{align*}
L^{\bfq, \infty}(i) = \begin{cases} L(i) & \textup{if }i \leq \#\bfp\\
\infty & \textup{otherwise}.
\end{cases}
\end{align*}
There is a partial order \(\preceq^\infty\) on \(\Tabqinf\) given by setting \(U \preceq^\infty U'\) provided that \(U(i) \geq U'(i)\) for all \(i \in [1, \#\bfq]\).

\begin{Lemma}
\(\Tabqinf\) forms a distributive lattice under \(\preceq^\infty\), where 
\begin{align*}
(U \wedge U')(i) = \max\{ U(i), U'(i)\}
\qquad
\textup{and}
\qquad
(U \vee U')(i) = \min\{ U(i), U'(i)\}.
\end{align*}
\end{Lemma}
\begin{proof}
For \(U,U' \in \Tabqinf\), define \(T: [1, \#\bfq] \to \Z_{\geq 0} \cup \{\infty\}\) by
\(
T(i)= \max\{ U(i), U'(i)\}
\). Once we verify that \(T \in \Tabqinf\), it is clear that \(T = U \wedge U'\). Assume that \(T(m) \in \Z_{\geq 0}\). We must verify that \(\alpha := (\bfq_1, \ldots, \bfq_m, T|_{[1,m]})\) satisfies properties (i)--(iii) of Definition~\ref{deftab}, and is hence a \((G, b, \ell)\)-tableau. As \(U(m), U'(m) \in \Z_{\geq 0}\), we have by assumption that \(\bfp := ((\bfq_1, \ldots, \bfq_m), U|_{[1,m]})\) and \(\bfp' := ((\bfq_1, \ldots, \bfq_m), U|_{[1,m]}')\) are \((G,b,\ell)\)-tableaux. Thus, for \(i \leq j\), we have \(U(i) \leq U(j)\) and \(U'(i) \leq U'(j)\), so it follows that \(T(i) = \max\{U(i), U'(i)\} \leq \max\{U'(i), U'(j)\} = T(j)\), and thus (i) is satisfied by \(\alpha\). For (ii), assume that \(i<j\), and \(q_{i-1} = q_j\). Then \(U(i) < U(j)\) and \(U'(i) < U'(j)\), which similarly implies that \(T(i)< T(j)\). Finally, we have that \(T(m) + m = U(m) + m \leq \ell\) or \(T(m) + m = U'(m) + m \leq\ell\), so (iii) is satisfied by \(\alpha\), giving the result.

The proof that \(U\vee U'\) is well-defined is similar. Distributivity is immediate as well, since:
\begin{align*}
(U \wedge (U' \vee U''))(i)  &= \max\{U(i), \min\{ U'(i), U''(i)\}\}\\
&= \min\{ \max\{U(i), U'(i)\} , \max\{U(i), U''(i)\}\}
= ((U\wedge U') \vee (U \wedge U''))(i),
\end{align*}
which completes the proof.
\end{proof}

\subsection{Tight tableaux and join-irreducibility}

\begin{Lemma}\label{JIq}
Let \((\bfp, L) \in \Tablq\). Then \((\bfp, L)\) is a tight \((G,b,\ell)\)-tableau if and only if \(h^\bfq((\bfp, L))\) is join-irreducible in \(\Tabqinf\).
\end{Lemma}
\begin{proof}
We will use the characterizing property of tight tableaux from Lemma~\ref{tightimage}: 
that \((\bfp, L)\in \Tab\) is tight if and only if, for every \((\bfp, L') \in \Tab\) with \(L'(\#\bfp) = L(\#\bfp)\), we have \(L'(r) \leq L(r)\) for all \(r \in [1, \#\bfp]\).

\((\implies)\) Assume \((\bfp, L)\) is tight, \(L^{\bfq, \infty} = h^\bfq((\bfp, L))\), and \(L^{\bfq, \infty} = L' \vee L''\) for some \(L', L'' \in \Tabqinf\). As \(L^{\bfq, \infty}(t) = \infty\) for \(t > \#\bfp\), we have that \(L'(t) = L''(t) = \infty\) for \(t > \#\bfp\). Since \(L^{\bfq, \infty}(\#\bfp) = \min \{L'(\bfp), L''(\bfp)\}\), we may assume without loss of generality that \(L^{\bfq, \infty}(\#\bfp) = L'(\bfp) \in \Z_{\geq 0}\). Assume by way of contradiction that \(L^{\bfq, \infty} \neq L'\). Then there must be some \(t < \#\bfp\) such that \(L^{\bfq, \infty}(t) = L''(t) < L'(t)\). But then \(((\bfq_1, \ldots, \bfq_{\#\bfp}), L'|_{[1, \#\bfp]}) = (\bfp,L'|_{[1, \#\bfp]}) \) is a \((G,b,\ell)\)-tableau with \(L'|_{[1, \#\bfp]}(\#\bfp) = L(\#\bfp)\), but  \(L'|_{[1, \#\bfp]}(t) > L(t)\), a contradiction of the tightness characterization of \((\bfp, L)\). 

\((\impliedby)\) We go by contrapositive. Assume that \((\bfp, L)\) is {\em not} tight. Then there exists a tight \((\bfp, M) \in \Tab\) such that \(M(\#\bfp) = L(\#\bfp)\), with some \(t < \#\bfp\) such that \(M(t) > L(t)\). Defining \(\widetilde{L^{\bfq, \infty}} \in \Tabqinf\) by \(\widetilde{L^{\bfq, \infty}}(\#\bfp) = \infty\) and \(\widetilde{L^{\bfq, \infty}}(k) = L^{\bfq, \infty}(k)\) otherwise, 
it is straightforward to check that 
\begin{align*}
L^{\bfq, \infty} = \widetilde{L^{\bfq, \infty}} \vee (M^\infty \wedge L^{\bfq, \infty}).
\end{align*}
Yet \(L^{\bfq, \infty} \neq \widetilde{L^{\bfq, \infty}}, (M^\infty \wedge L^{\bfq, \infty})\), so \(L^{\bfq, \infty}\) is not join-irreducible, completing the proof.
\end{proof}

For a poset \(\lambda\), we write \(J(\lambda) \subseteq \lambda\) for the poset of join-irreducible elements of \(\lambda\), with induced partial order.
For a \((G,b)\)-path \(\bfq\), write 
\begin{align*}
\IPGbl^{\preceq \bfq} := \{ \la \bfp, a \ra \in \IPGbl \mid \bfp \preceq^\textup{pre} \bfq\}.
\end{align*}

\begin{Lemma}\label{posetbij}
The map \(h^\bfq \circ \tau : \IPGbl^{\preceq \bfq} \to J(\Tabqinf)\) is an isomorphism of posets.
\end{Lemma}
\begin{proof}
That \(h^\bfq \circ \tau\) is a bijection of sets follows from Lemmas~\ref{tightimage} and \ref{JIq}. We check that \(h^\bfq \circ \tau\) and its inverse are order-preserving. Let \(\la \bfp, a \ra \preceq^\textup{IP} \la \bfr, b \ra\in \IPGbl^{\preceq \bfq}\). Then \(\#\bfp \leq \#\bfr\) and \(n_\bfp + a \geq d_\bfr(\#\bfp) + b\). We have 
\begin{align*}
L_{\la \bfp, a \ra}^{ \bfq, \infty}(y) 
&= 
\begin{cases}
d_\bfp(y) + a -1 &\textup{if }y \in [1, \#\bfp];\\
\infty & \textup{if }y \in [\#\bfp +1, \#\bfq],
\end{cases}\\
L_{\la \bfr, b \ra}^{ \bfq, \infty}(y) 
&= 
\begin{cases}
d_\bfr(y) + b -1 &\textup{if }y \in [1, \#\bfr];\\
\infty & \textup{if }y \in [\#\bfr +1, \#\bfq].
\end{cases}
\end{align*} 
Then, since \(\#\bfp \leq \#\bfr\), we have \(L_{\la \bfp, a \ra}^{ \bfq, \infty}(y) \geq L_{\la \bfr, b \ra}^{ \bfq, \infty}(y)\) for all \(y \in [\#\bfp + 1, \#\bfq]\). For \(y \in [1, \#\bfp]\), we have
\begin{align*}
L_{\la \bfp, a \ra}^{ \bfq, \infty}(y) &= d_\bfp(y) + a - 1 \geq d_\bfp(y) + d_\bfr(\#\bfp) + b - n_\bfp -1\\
&\geq d_\bfr(\#\bfp) + b -1 \geq d_\bfr(y) + b -1 = L_{\la \bfr, b \ra}^{ \bfq, \infty}(y),
\end{align*}
where the second inequality follows from the fact that \(n_\bfp \geq d_\bfp(y)\) for all \(y \in [1,\#\bfp]\), and the third inequality follows from the fact that \(d_\bfr(\#\bfp) \geq d_\bfr(y)\) for all \(y \in [1,\#\bfp]\). Therefore \(f_\bfq \circ \tau\) is order-preserving. 

In the other direction, assume that \(L_{\la \bfp, a \ra}^{ \bfq, \infty} \preceq^{\bfq, \infty} L_{\la \bfr, b \ra}^{ \bfq, \infty}\). Then
\begin{align*}
 n_\bfp + a -1 =L_{\la \bfp, a \ra}^{ \bfq, \infty}(\#\bfp) \geq L_{\la \bfr, b \ra}^{ \bfq, \infty}(\#\bfp) = d_\bfr(\#\bfp) + b -1,
\end{align*}
so \(n_\bfp + a \geq d_\bfr(\#\bfp) + b\), and thus \(\la \bfp, a \ra \preceq^\textup{IP} \la \bfr, b \ra\), so \((h^\bfq \circ \tau)^{-1}\) is order-preserving as well, completing the proof.
\end{proof}

%%%%%%%%%%%%%%%%%%%%%%%%%%%%%%%%%%%%%%

\section{Robotic arms}\label{robotarms}
Now we define, in more technical fashion, the robotic arm and configuration space setup informally described in \S\ref{robint}.

\subsection{The robotic arm workspace}
Continuing with \(G,b,\ell\) fixed, define the {\em workspace graph} \(W_G = (V_{W}, E_{W})\) with vertices \(V_{W_G} = V_G \times \Z_{\geq 0}\), and edges \(E_{W_G} = E_{W_G}^\textup{hor} \cup E_{W_G}^\textup{ver}\), where:
\begin{enumerate}
\item \(E_{W_G}^\textup{hor} =\{\{ (v,h),(w,h)  \} \mid \{v,w\} \in E_G, h \in \Z_{\geq 0}\}\) are the `horizontal' edges, and;
\item \(E_{W_G}^\textup{ver} = \{\{(v,h),(v,h+1)\} \mid v \in V_G, h \in \Z_{\geq 0}\} \) are the `vertical' edges.
\end{enumerate}
For \(\alpha = (v,h) \in V_{W_G}\), we write \(G(\alpha) := v\), \(\height(\alpha):=h\) and refer to \(h\) as the `height' of \(\alpha\).
We depict \(W_G\) as a skyscraper whose floors are horizontal copies of the graph \(G\), with \(G \times \{0\}\) being viewed as the `ground floor', and with vertical edges connecting each vertex with its copy on the next floor.

\subsection{The robotic arm}
Informally, the {\em robotic arm} \(\mathcal{R}_{G,b,\ell}\) {\em of length \(\ell\), based at \(b\) in the underlying graph \(G\)} is a non-self-intersecting linked sequence of \(\ell\) line segments, anchored on the ground floor of \(S\) at \((b,0)\), with the arm extending along edges in the workspace \(W_G\) either upwards or horizontally within \(G\)-floors. The robotic arm is capable of performing moves which alter its configuration; see \S\ref{movecat}. We make this more precise in what follows.

\begin{Definition}
An {\em \(\mathcal{R}_{G, b, \ell}\)-configuration} (or when there is no chance of confusion, just a {\em configuration}) \(\bfx\) is a \((W_G,(b,0))\)-path of length \(\ell\), with weakly increasing height: \(\height(x_0) \leq \cdots \leq \height(x_\ell)\). We write \(\mathcal{V}_{G,b,\ell}\) for the set of all \(\mathcal{R}_{G,b,\ell}\)-configurations. We call the fully vertical configuration \(\bfb\) with vertex sequence \((b,0), (b,1), \ldots, (b,\ell)\) 
the {\em initial configuration}.
\end{Definition}

A visual example of a graph, workspace, and configuration is shown in Figure~\ref{fig:configex}.

\begin{figure}[h]
\begin{align*}
\hackcenter{
\begin{tikzpicture}[scale=0.6]
\draw[white, fill =white]  (0,0) circle (10pt);
\draw[black, fill =black]  (0,0) circle (3pt);
\draw[black, fill =black]  (3,1) circle (3pt);
\draw[black, fill =black]  (-1,2) circle (3pt);
\draw[black, fill =black]  (5,1) circle (3pt);
\draw[ thick, join=round, cap=round] (0,0)--(3,1)--(-1,2)--(0,0);
\draw[ thick, join=round, cap=round] (3,1)--(5,1);
\draw[ thick, join=round, cap=round, white] (-1,2)--(-1,8.5);
  \node[left] at (0,0){ $\scriptstyle b$};    
   \node[below] at (3,1){ $\scriptstyle a$};    
    \node[left] at (-1,2){ $\scriptstyle c$};    
     \node[right] at (5,1){ $\scriptstyle d$};    
%%%
%%%
        % \node[] at (-2-8,3.5){ $\scriptstyle{\mathbf{3}}$};      
\end{tikzpicture}
}
\;\;\;
\hackcenter{
\begin{tikzpicture}[scale=0.6]
\draw[white, fill =white]  (0,0) circle (10pt);
\draw[black, fill =black]  (0,0) circle (3pt);
\draw[black, fill =black]  (3,1) circle (3pt);
\draw[black, fill =black]  (-1,2) circle (3pt);
\draw[black, fill =black]  (5,1) circle (3pt);
  \node[left] at (0,0){ $\scriptstyle (b,0)$};    
    \node[right] at (0,1){ $\scriptstyle (b,1)$};    
        \node[right] at (0,1+1.3){ $\scriptstyle (b,2)$};   
          \node[right] at (0,1+2.6){ $\scriptstyle (b,3)$};   
            \node[right] at (0,1+3.9){ $\scriptstyle (b,4)$};   
   \node[below] at (3,1){ $\scriptstyle (a,0)$};    
      \node[right] at (2.9,1.9){ $\scriptstyle (a,1)$}; 
        \node[right] at (2.9,1.9+1.3){ $\scriptstyle (a,2)$}; 
         \node[right] at (2.9,1.9+2.6){ $\scriptstyle (a,3)$}; 
          \node[right] at (2.9,1.9+3.9){ $\scriptstyle (a,4)$}; 
    \node[left] at (-1,2){ $\scriptstyle (c,0)$};    
       \node[left] at (-1,2+1.3){ $\scriptstyle (c,1)$};  
        \node[left] at (-1,2+2.6){ $\scriptstyle (c,2)$};  
                \node[left] at (-1,2+3.9){ $\scriptstyle (c,3)$};  
                        \node[left] at (-1,2+5.2){ $\scriptstyle (c,4)$};  
     \node[right] at (5,1){ $\scriptstyle (d,0)$};    
         \node[right] at (5,1+1.3){ $\scriptstyle (d,1)$};  
             \node[right] at (5,1+2.6){ $\scriptstyle (d,2)$};  
               \node[right] at (5,1+3.9){ $\scriptstyle (d,3)$};  
                 \node[right] at (5,1+5.2){ $\scriptstyle (d,4)$};  
\draw[ thick, join=round, cap=round] (0,0)--(3,1)--(-1,2)--(0,0);
\draw[ thick, join=round, cap=round] (0,0+1.3)--(3,1+1.3)--(-1,2+1.3)--(0,0+1.3);
\draw[ thick, join=round, cap=round] (0,0+2.6)--(3,1+2.6)--(-1,2+2.6)--(0,0+2.6);
\draw[ thick, join=round, cap=round] (0,0+3.9)--(3,1+3.9)--(-1,2+3.9)--(0,0+3.9);
\draw[ thick, join=round, cap=round] (0,0+5.2)--(3,1+5.2)--(-1,2+5.2)--(0,0+5.2);
\draw[ thick, join=round, cap=round] (0,0+1.3)--(3,1+1.3)--(-1,2+1.3)--(0,0+1.3);
\draw[ thick, join=round, cap=round] (3,1)--(5,1);
\draw[ thick, join=round, cap=round] (3,1+1.3)--(5,1+1.3);
\draw[ thick, join=round, cap=round] (3,1+2.6)--(5,1+2.6);
\draw[ thick, join=round, cap=round] (3,1+3.9)--(5,1+3.9);
\draw[ thick, join=round, cap=round] (3,1+5.2)--(5,1+5.2);
\draw[ thick, join=round, cap=round] (3,1)--(3,8);
\draw[ thick, join=round, cap=round] (5,1)--(5,8);
\draw[ thick, join=round, cap=round] (0,0)--(0,7.5);
\draw[ thick, join=round, cap=round] (-1,2)--(-1,8.5);
%%%
   \node[] at (1.5, 8){ $\vdots$};       
\end{tikzpicture}
}
\;\;\;
\hackcenter{
\begin{tikzpicture}[scale=0.6]
    \node[left] at (-.2,0){ $\scriptstyle (b,0)$};    
   \node[below] at (3,0.9){ $\scriptstyle (a,0)$};    
    \node[left] at (-1,2){ $\scriptstyle (c,0)$};    
     \node[right] at (5,1){ $\scriptstyle (d,0)$};    
\draw[ thick, join=round, cap=round] (0,0)--(3,1)--(-1,2)--(0,0);
\draw[black, fill =black]  (0,0) circle (3pt);
\draw[black, fill =black]  (3,1) circle (3pt);
\draw[black, fill =black]  (-1,2) circle (3pt);
\draw[black, fill =black]  (5,1) circle (3pt);
\draw[ thick, join=round, cap=round] (0,0+1.3)--(3,1+1.3)--(-1,2+1.3)--(0,0+1.3);
\draw[ thick, join=round, cap=round] (0,0+2.6)--(3,1+2.6)--(-1,2+2.6)--(0,0+2.6);
\draw[ thick, join=round, cap=round] (0,0+3.9)--(3,1+3.9)--(-1,2+3.9)--(0,0+3.9);
\draw[ thick, join=round, cap=round] (0,0+5.2)--(3,1+5.2)--(-1,2+5.2)--(0,0+5.2);
\draw[ thick, join=round, cap=round] (0,0+1.3)--(3,1+1.3)--(-1,2+1.3)--(0,0+1.3);
\draw[ thick, join=round, cap=round] (3,1)--(5,1);
\draw[ thick, join=round, cap=round] (3,1+1.3)--(5,1+1.3);
\draw[ thick, join=round, cap=round] (3,1+2.6)--(5,1+2.6);
\draw[ thick, join=round, cap=round] (3,1+3.9)--(5,1+3.9);
\draw[ thick, join=round, cap=round] (3,1+5.2)--(5,1+5.2);
\draw[ thick, join=round, cap=round] (3,1)--(3,8);
\draw[ thick, join=round, cap=round] (5,1)--(5,8);
\draw[ thick, join=round, cap=round] (0,0)--(0,7.5);
\draw[ thick, join=round, cap=round] (-1,2)--(-1,8.5);
\draw[line width = 1.3mm, join=round, cap=round, red] (3,3.6)--(-1,4.6);
%%%
%%%
\draw[blue, fill =blue]  (0,0) circle (10pt);
\draw[line width = 1.3mm, join=round, cap=round, red] (0,0)--(3,1)--(3,2.3)--(5,2.3)--(5,3.6)--(3,3.6);
\draw[ thick, join=round, cap=round] (0,3.9)--(-1,5.9);
\draw[ thick, join=round, cap=round] (0,3.9)--(0,5.2);
\draw[line width = 2.1mm, join=round, cap=round, white, shorten >=0.5cm] (0,3.9)--(3,4.9);
\draw[line width = 1.3mm, join=round, cap=round, ->, red] (-1,4.6)--(0,2.6)--(0,3.9)--(3,4.9);
\draw[red, fill =red]  (0,0) circle (7pt);
\draw[red, fill =red]  (3,1) circle (7pt);
\draw[red, fill =red]  (3,2.3) circle (7pt);
\draw[red, fill =red]  (5,2.3) circle (7pt);
\draw[red, fill =red]  (5,3.6) circle (7pt);
\draw[red, fill =red]  (3,3.6) circle (7pt);
\draw[red, fill =red]  (-1,4.6) circle (7pt);
\draw[red, fill =red]  (0,2.6) circle (7pt);
\draw[red, fill =red]  (0,3.9) circle (7pt);
       \node[] at (1.5, 8){ $\vdots$};      
\end{tikzpicture}
}
\end{align*}
\caption{A graph \(G\); the associated workspace \(W_G\); the configuration \(\bfx \in \mathcal{V}_{G,b,9}\) with vertex sequence \((b,0),(a,0),(a,1),(d,1),(d,2),(a,2),(c,2),(b,2),(b,3),(a,3)\).}
\label{fig:configex}       % Give a unique label
\end{figure}
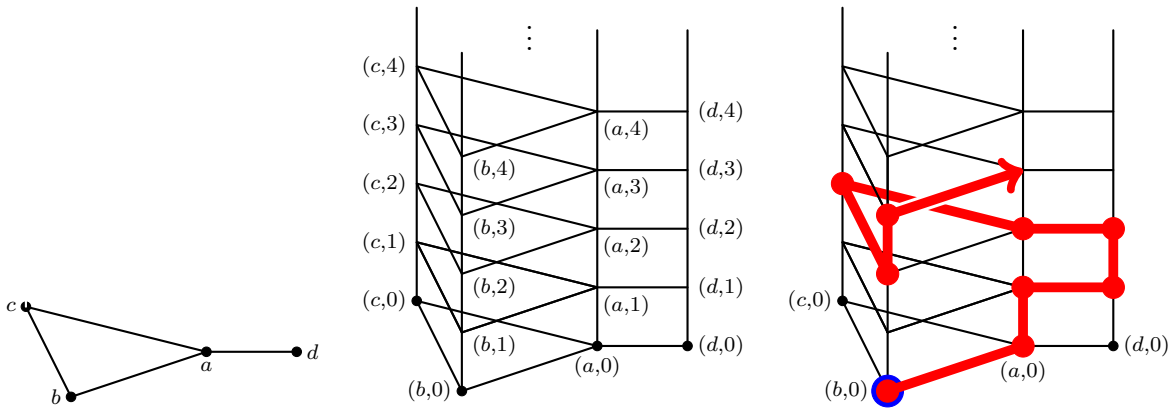

\subsection{Move catalogue}\label{movecat}
Let \(v,w \in V_G\) be such that \(\{v,w\} \in E_G\), and \(h \in \Z_{\geq 0}\). 
We define now a number of operators \(\{\mathsf{T}^{\pm 1}_{v,w,h}, \mathsf{C}^{\pm 1}_{v,w,h}\}\), each defined on a support subset of \(\mathcal{V}_{G,b,\ell}\), which locally alter configurations, injectively carrying configurations in the support back into \(\mathcal{V}_{G,b,\ell}\). See Figure~\ref{fig:moves} for a visual representation of the moves defined in \S\ref{tmm}, \ref{cmm} below.

\subsubsection{Tail moves}\label{tmm}  Set \(\textup{supp}(\mathsf{T}^1_{v,w,h}) \subseteq \mathcal{V}_{G,b,\ell}\) to be the set of all \(\bfx \in \mathcal{V}_{G,b,\ell}\) such that \(\bfx_\ell = \{(v,h),(w,h)\}\). Then define
\(
\mathsf{T}^1_{v,w,h}\bfx \in \mathcal{V}_{G,b,\ell}
\) by replacing \(\bfx_\ell = \{(v,h),(w,h)\}\) in \(\bfx\) with \(\{(v,h),(v,h+1)\}\).
We call \(\mathsf{T}^1_{v,w,h}\) an {\em upward tail move}.
We set \(\mathsf{T}^{-1}_{v,w,h}\) to be inverse to \(\mathsf{T}_{v,w,h}\), defined on \(\textup{supp}(\mathsf{T}^{-1}_{v,w,h}) = \mathsf{T}^1_{v,w,h}(\textup{supp}(\mathsf{T}^1_{v,w,h}))\). We call \(\mathsf{T}^{-1}_{v,w,h}\) a {\em downward tail move}.

\subsubsection{Corner moves}\label{cmm} Set \(\textup{supp}(\mathsf{C}^{1}_{v,w,h}) \subseteq \mathcal{V}_{G,b,\ell}\) to be the set of all \(\bfx \in \mathcal{V}_{G,b,\ell}\) such that \(\bfx_j = \{(v,h),(w,h)\}\) and \(\bfx_{j+1} = \{(w,h),(w,h+1)\}\) for some \(j\), and \(x_i \neq (v,h+1)\) for all \(i \in [0,\ell]\). Then define
\(
\mathsf{C}^1_{v,w,h}\bfx \in \mathcal{V}_{G,b,\ell}
\) by replacing \(\{(v,h),(w,h)\}\), \(\{(w,h),(w,h+1)\}\) in \(\bfx\) with \(\{(v,h),(v,h+1)\}\), \( \{(v,h+1),(w,h+1)\}\), respectively.
We call \(\mathsf{C}^1_{v,w,h}\) an {\em upward corner move}.
We set \(\mathsf{C}^{-1}_{v,w,h}\) to be inverse to \(\mathsf{C}_{v,w,h}\), defined on \(\textup{supp}(\mathsf{C}^{-1}_{v,w,h}) = \mathsf{C}^1_{v,w,h}(\textup{supp}(\mathsf{C}^1_{v,w,h}))\). We call \(\mathsf{C}^{-1}_{v,w,h}\) a {\em downward corner move}.

\begin{figure}[h]
\begin{align*}
\hackcenter{
\begin{tikzpicture}[scale=0.6]
\draw[ thick, join=round, cap=round] (0,0)--(2,0)--(2,2)--(0,2)--(0,0);
\draw[black, fill =black]  (0,0) circle (3pt);
\draw[black, fill =black]  (2,2) circle (3pt);
\draw[black, fill =black]  (0,2) circle (3pt);
\draw[black, fill =black]  (2,0) circle (3pt);
    \node[below] at (0,-.2){ $\scriptstyle(v,h)$};  
     \node[below] at (2,-.2){ $\scriptstyle(w,h)$};  
        \node[above] at (0,2.2){ $\scriptstyle(v,h+1)$};  
     \node[above] at (2,2.2){ $\scriptstyle(w,h+1)$};  
\draw[red, fill =red]  (0,0) circle (7pt);
\draw[line width = 1.3mm, join=round, cap=round, ->, red] (0,0)--(2,0);
 \draw[thick, ->] (3,1.1) .. controls ++(.45,0.45) and ++(-.45,0.45) .. (5,1.1);
  \draw[thick, ->] (5,0.9) .. controls ++(-.45,-0.45) and ++(.45,-0.45) .. (3,0.9);
  \node[] at (4,2){$\scriptstyle\mathsf{T}^{+1}_{v,w,h}$};
    \node[] at (4,0){$\scriptstyle\mathsf{T}^{-1}_{v,w,h}$};
  \draw[ thick, join=round, cap=round] (0+6,0)--(2+6,0)--(2+6,2)--(0+6,2)--(0+6,0);
\draw[black, fill =black]  (0+6,0) circle (3pt);
\draw[black, fill =black]  (2+6,2) circle (3pt);
\draw[black, fill =black]  (0+6,2) circle (3pt);
\draw[black, fill =black]  (2+6,0) circle (3pt);
    \node[below] at (0+6,-.2){ $\scriptstyle(v,h)$};  
     \node[below] at (2+6,-.2){ $\scriptstyle(w,h)$};  
        \node[above] at (0+6,2.2){ $\scriptstyle(v,h+1)$};  
     \node[above] at (2+6,2.2){ $\scriptstyle(w,h+1)$};  
\draw[red, fill =red]  (0+6,0) circle (7pt);
\draw[line width = 1.3mm, join=round, cap=round, ->, red] (0+6,0)--(0+6,2);
\end{tikzpicture}
}
\qquad
\qquad
\hackcenter{
\begin{tikzpicture}[scale=0.6]
\draw[ thick, join=round, cap=round] (0,0)--(2,0)--(2,2)--(0,2)--(0,0);
\draw[black, fill =black]  (0,0) circle (3pt);
\draw[black, fill =black]  (2,2) circle (3pt);
\draw[black, fill =black]  (0,2) circle (3pt);
\draw[black, fill =black]  (2,0) circle (3pt);
    \node[below] at (0,-.2){ $\scriptstyle(v,h)$};  
     \node[below] at (2,-.2){ $\scriptstyle(w,h)$};  
        \node[above] at (0,2.2){ $\scriptstyle(v,h+1)$};  
     \node[above] at (2,2.2){ $\scriptstyle(w,h+1)$};  
\draw[red, fill =red]  (0,0) circle (7pt);
\draw[red, fill =red]  (2,0) circle (7pt);
\draw[red, fill =red]  (2,2) circle (7pt);
\draw[line width = 1.3mm, join=round, cap=round, red] (0,0)--(2,0)--(2,2);
 \draw[thick, ->] (3,1.1) .. controls ++(.45,0.45) and ++(-.45,0.45) .. (5,1.1);
  \draw[thick, ->] (5,0.9) .. controls ++(-.45,-0.45) and ++(.45,-0.45) .. (3,0.9);
  \node[] at (4,2){$\scriptstyle\mathsf{C}^{+1}_{v,w,h}$};
    \node[] at (4,0){$\scriptstyle\mathsf{C}^{-1}_{v,w,h}$};
  \draw[ thick, join=round, cap=round] (0+6,0)--(2+6,0)--(2+6,2)--(0+6,2)--(0+6,0);
\draw[black, fill =black]  (0+6,0) circle (3pt);
\draw[black, fill =black]  (2+6,2) circle (3pt);
\draw[black, fill =black]  (0+6,2) circle (3pt);
\draw[black, fill =black]  (2+6,0) circle (3pt);
    \node[below] at (0+6,-.2){ $\scriptstyle(v,h)$};  
     \node[below] at (2+6,-.2){ $\scriptstyle(w,h)$};  
        \node[above] at (0+6,2.2){ $\scriptstyle(v,h+1)$};  
     \node[above] at (2+6,2.2){ $\scriptstyle(w,h+1)$};  
\draw[red, fill =red]  (0+6,0) circle (7pt);
\draw[red, fill =red]  (0+6,2) circle (7pt);
\draw[red, fill =red]  (2+6,2) circle (7pt);
\draw[line width = 1.3mm, join=round, cap=round, red] (0+6,0)--(0+6,2)--(2+6,2);
\end{tikzpicture}
}
\end{align*}
\caption{Tail move; corner move.}
\label{fig:moves}       % Give a unique label
\end{figure}
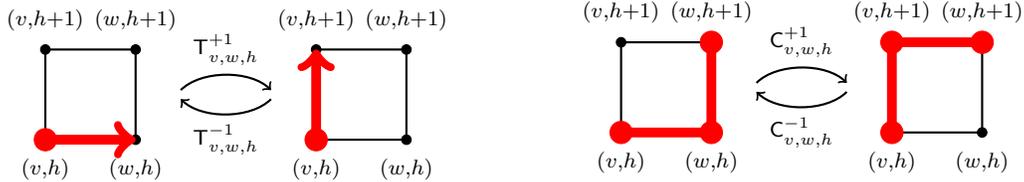

\subsubsection{Legal moves} We call \(\mathscr{M} = \{\mathsf{T}^{\pm 1}_{v,w,h}, \mathsf{C}^{\pm 1}_{v,w,h} \mid v,w \in V_G, \;\{v,w\} \in E_G, \;h \in \Z_{\geq 0}\}\) the set of {\em moves}. For \(\bfx \in \mathcal{V}_{G,b,\ell}\), we call \(\mathscr{L}_{\bfx}:=\{ \mathsf{M} \in \mathscr{M} \mid  \bfx \in \textup{supp}(\mathsf{M})\}\) the set of {\em legal moves for} \(\bfx\). If \(\mathscr{A} \subseteq \mathscr{L}_{\bfx}\), then we abuse notation and write \(\bfx \in \textup{supp}(\mathscr{A})\).

\subsubsection{Commutative moves} Let \(\mathscr{A} \subseteq \mathscr{M}\). We say \(\mathscr{A}\) is a {\em commutative set of moves} provided that it satisfies the following property: For any \(\bfx \in \textup{supp}(\mathscr{A})\), and any
subset \(\mathscr{S} = \{\mathsf{S}_1, \ldots, \mathsf{S}_{|\mathscr{S}|}\} \subseteq \mathscr{A}\), the moves in \(\mathscr{S}\) may be legally applied to \(\bfx\) in any order, with the resulting state dependent only on \(\mathscr{S}\), and not on the order of application.
I.e., for every \(j\)-permutation \((i_1, \ldots, i_j)\) of \([1, |\mathscr{S}|]\), we have that \(\mathsf{S}_{i_{j-1}} \cdots \mathsf{S}_{i_1} \bfx \in \textup{supp}(\mathsf{S}_{i_j})\) and \(\mathsf{S}_{i_{j}} \cdots \mathsf{S}_{i_1} \bfx = \mathsf{S}_{t_{j}} \cdots \mathsf{S}_{t_1} \bfx\) for any other permutation \((t_1, \ldots, t_j)\) of \(\{i_1, \ldots, i_j\}\).

Let \(\mathscr{A}\) be a commutative set of moves. If \(\mathscr{S} = \{\mathsf{S}_1, \ldots, \mathsf{S}_{|\mathscr{S}|}\} \subseteq \mathscr{A}\) and \(\bfx \in \textup{supp}(\mathscr{A})\), we may write \(\mathscr{S}\bfx := \mathsf{S}_{|\mathscr{S}|} \cdots \mathsf{S}_1 \bfx  \), since commutativity removes any ambiguity in this assignment. It follows from definitions that \(\mathscr{S} \bfx \neq \mathscr{S}' \bfx\) for every distinct \(\mathscr{S}, \mathscr{S}' \subseteq \mathscr{A}\), so that \(|\{\mathscr{S} \bfx \mid \mathscr{S} \subseteq \mathscr{A}\}| = 2^{|\mathscr{A}|}\).

\begin{Example}
Consider the configuration \(\bfx \in \mathcal{V}_{G,b,9}\) in Figure~\ref{fig:configex}. The set of legal moves for \(\bfx\) are:
\begin{align*}
\mathcal{L}_\bfx = \{ \mathsf{C}^{+1}_{b,a,0}, \mathsf{C}^{-1}_{a,d,0}, \mathsf{C}^{+1}_{c,b,2}, \mathsf{T}^{+1}_{b,a,3} \}.
\end{align*}
The subsets of \(\mathcal{L}_\bfx\) which are commutative are those which do not contain both \(\mathsf{C}^{+1}_{b,a,0}\) and \( \mathsf{C}^{-1}_{a,d,0}\).
\end{Example}

\subsection{Transition graph} The {\em transition graph \(\mathcal{T}_{G,b,\ell}\) of the robotic arm \(\mathcal{R}_{G,b,\ell}\)} is defined to have vertices \(\mathcal{V}_{G,b,\ell}\), with edges \(\{\{\bfx, \mathsf{M}\bfx\} \mid \mathsf{M} \in \mathscr{M}, \bfx \in \textup{supp}(\mathsf{M})\}\) corresponding to legal moves.

\begin{Lemma}
The transition graph  \(\mathcal{T}_{G,b,\ell}\) is connected.
\end{Lemma}
\begin{proof}
We prove that every \(\bfx \in \mathcal{V}_{G,b,\ell}\) is connected to the initial configuration \(\bfb\). We go by induction on the number \(m\) of horizontal edges in \(\bfx\), the base case \(m=0\) being clear. Now make the induction assumption and assume that \(\bfx_j = \{(v,h),(w,h)\}\) is the last horizontal edge in \(\bfx\), so that \(\bfx_{j+i} = \{(w,h+i-1), (w,h+i)\}\) for \(i = 1, \ldots \ell - j\). Then \(\bfx\) is connected to \[\bfy :=\mathsf{T}^1_{v,w,h+\ell - j}\mathsf{C}^1_{v,w,h+\ell-j-1}\cdots \mathsf{C}^1_{v,w,h+1}\mathsf{C}^1_{v,w,h}\bfx\]
in \(\mathcal{T}_{G,b,\ell}\), 
with \(\bfy_k = \bfx_k \) for \(k = 1, \ldots, j-1\), and \(\bfy_{j+i} = \{(v,h+i),(v,h+i+1)\}\) for \(i = 0, \ldots, \ell -j + 1\). Therefore \(\bfy\) has fewer horizontal segments than \(\bfx\), and thus is connected by the induction assumption  to \(\bfb\), completing the proof.
\end{proof}

\subsection{Configuration space}\label{robcube} The {\em configuration space \(\mathcal{S}_{G,b,\ell}\) of the robotic arm \(\mathcal{R}_{G,b,\ell}\)} is the cubical complex defined as follows. The 0-skeleton of \(\mathcal{S}_{G,b,\ell}\) is the set of vertices \(\mathcal{V}_{G,b,\ell}\). Cubes are added as follows. Assume that \(\bfY\) is a set of \(2^k\) configurations, and that there is a set of \(k\) commutative moves \(\mathscr{A}\) and \(\bfx \in \textup{supp}(\mathscr{A})\) such that \(\bfY = \{\mathscr{S} \bfx \mid \mathscr{S} \subseteq \mathscr{A}\}\). Then \(\bfY\) forms the vertices of a \(k\)-cube in \(\mathcal{S}_{G,b,\ell}\), which we label by \([\mathscr{A}; \bfx]\). The boundary of this \(k\)-cube is the collection of \(2k\) faces
\(
\bigcup_{\mathsf{M} \in \mathscr{A}} 
[\mathscr{A} \backslash \mathsf{M}; \bfx] \cup
[\mathscr{A} \backslash \mathsf{M}; \mathsf{M} \bfx].
\)
We endow \(\mathcal{S}_{G,b,\ell}\) with a Euclidean metric by letting each \(k\)-cube be a unit cube.

\begin{Example}
In Figure~\ref{fig:fullspace} we show the full configuration space \(\mathcal{S}_{C_3, b, 5}\) for the robotic arm of length five over the the cycle graph \(C_3\) with three vertices. The initial configuration \(\bfb\) is indicated in green. The \(0\)-skeleton is the set of configurations \(\mathcal{V}_{C_3, b, 5}\). The \(1\)-skeleton is the transition graph \(\mathcal{T}_{C_3, b,5}\). There at most \(3\) commutative moves from any given configuration in this setting, so \(\mathcal{S}_{C_3, b, 5}\) contains no \(k\)-cubes for \(k \geq 4\).
\end{Example}

\subsection{Connecting configurations and path tableaux}

Let \(\bfx \in \mathcal{V}_{G,b,\ell}\). Let \(\{i_1 < \cdots < i_t \} \subseteq [1, \#\bfx]\) be the subset of indices of horizontal edges in \(\bfx\). Then it follows that we have a \((G,b)\)-path \(\overline \bfx\) given by the edges \((G(\bfx_{i_1}), \ldots, G(\bfx_{i_t}))\), and a labeling function \(L_{\bfx}:[1,t] \to \Z_{\geq 0}\) defined by \(L_{\bfx}(r) = \height(\bfx_{i_r})\).

\begin{Lemma}\label{VTabbij}
The assignment \(\bfx \mapsto (\overline \bfx, L_\bfx)\) gives a well-defined bijection \(f: \mathcal{V}_{G,b,\ell} \to \Tab\).
\end{Lemma}
\begin{proof}
We check that (i)--(iii) of Definition~\ref{deftab} are satisfied by \((\overline \bfx, L_\bfx)\). This is straightforward: (i) follows from the fact that the vertices of \(\bfx\) are of weakly increasing height, (ii) follows from the fact that \(\bfx\) is non-self intersecting; (iii) follows from the fact that 
\(L(\#\overline{\bfx})\) is the height of the highest horizontal segment in \(\bfx\) (and hence less than or equal to the number of vertical segments in \(\bfx\)), \(\#\overline \bfx\) is the number of horizontal segments in \(\bfx\), and \(\ell\) is the total number of segments in \(\bfx\).

To see that \(f\) is a bijection, we manually construct an inverse map. Let \((\bfp, L) \in \Tab\). Define an associated sequence \(\bfy\) of horizontal edges in \(W\) by setting: \(\bfy_i = \{(p_{i-1}, L(i)), (p_i, L(i))\}\) for \(i \in [1, \#\bfp]\). There is a unique way to complete \(\bfy\) to an element \(\hat \bfy \in \mathcal{V}_{G,b,\ell}\) by inserting only vertical edges. Specifically, noting that \(p_0 = b\), we have:
\begin{align*}
\hat \bfy = &( \{(p_0,0), (p_0,1)\}, \ldots, \{(p_0,L(1)-1), (p_0,L(1))\}, \{(p_0, L(1)), (p_1, L(1))\},\\
&\hspace{5mm} \{(p_1,L(1)), (p_1,L(1) + 1)\}, \ldots, \{(p_1,L(2)-1), (p_1,L(2))\}, \{(p_1, L(2)), (p_2, L(2))\},\\
&\hspace{10mm} \ldots\\
&\hspace{10mm}
\{(p_{\#\bfp}, L(\#\bfp)),(p_{\#\bfp}, L(\#\bfp)+1) \}, \ldots, \{(p_{\#\bfp}, \ell - \#\bfp-1),(p_{\#\bfp}, \ell - \#\bfp) \}).
\end{align*}
It is straightforward to check that this is a well-defined \((W,\bfb)\)-path of length \(\ell\) and weakly increasing height. That it is non-self-intersecting follows from Definition~\ref{deftab}(ii). Thus the assignment \((\bfp, L) \mapsto \hat \bfy\) defines a function \(\Tab \to \mathcal{V}_{G,b,\ell}\) which is easily checked to be a mutual inverse for \(f\).
\end{proof}

\begin{landscape}
\begin{figure}
{}
\vspace{2cm}
{}
\includegraphics[width=22.5cm]{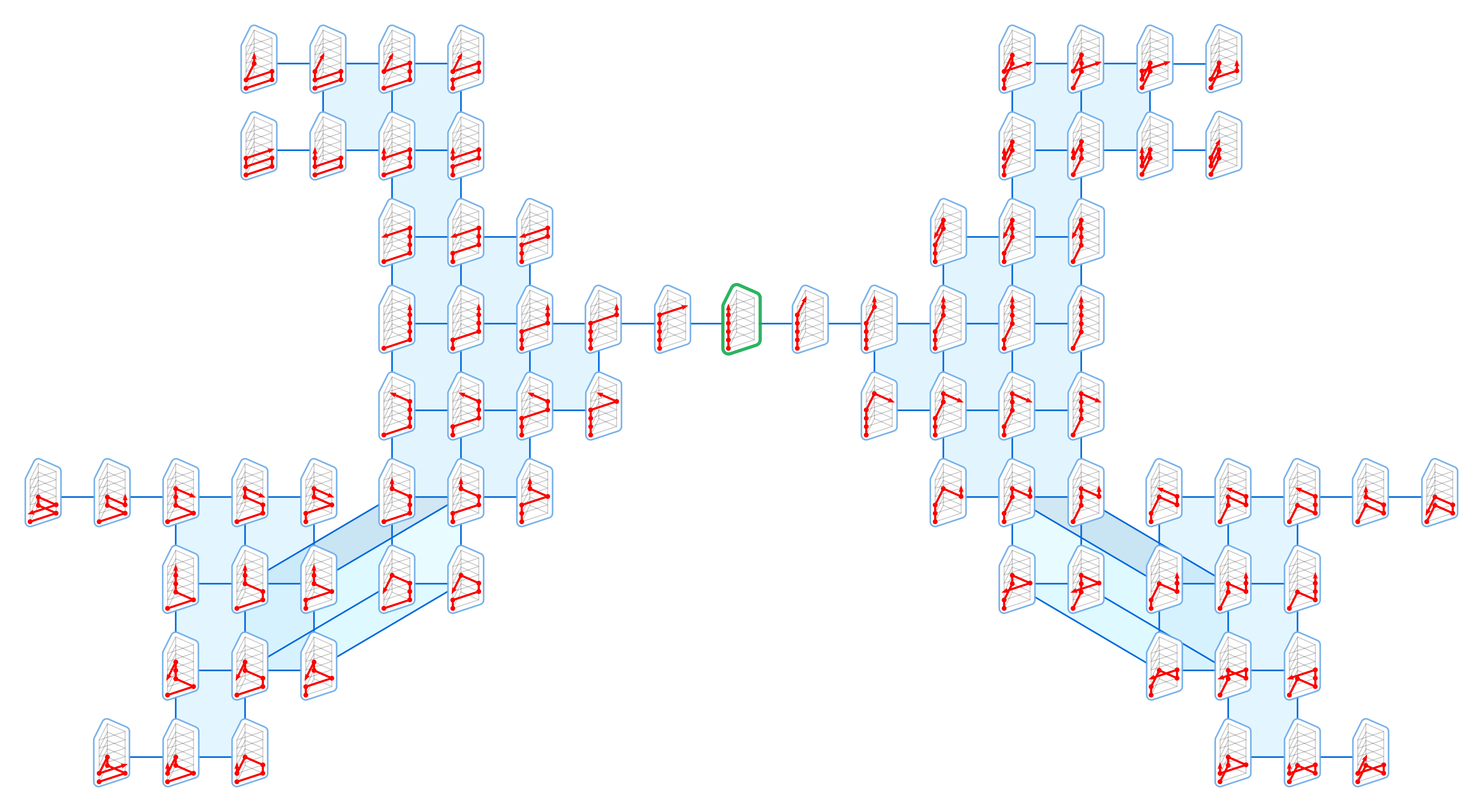}
\caption{The robotic arm configuration space \(\mathcal{S}_{C_3, b, 5}\). 
}
\label{fig:fullspace}       % Give a unique label
\end{figure}
\end{landscape}

\begin{Example}
In Figure~\ref{fig:configtotab} we show a visual depiction of the bijection  \(f: \mathcal{V}_{G,b,\ell} \to \Tab\). One can think of this bijection as compressing the vertical segments and recording only the heights of the horizontal segments of the configuration as labels on the resulting \(G\)-path.
\end{Example}

\begin{figure}[h]
\begin{align*}
\hackcenter{
\begin{tikzpicture}[scale=0.6] 
\draw[ thick, join=round, cap=round] (0,0)--(3,1)--(-1,2)--(0,0);
\draw[black, fill =black]  (0,0) circle (3pt);
\draw[black, fill =black]  (3,1) circle (3pt);
\draw[black, fill =black]  (-1,2) circle (3pt);
\draw[black, fill =black]  (5,1) circle (3pt);
\draw[ thick, join=round, cap=round] (0,0+1.3)--(3,1+1.3)--(-1,2+1.3)--(0,0+1.3);
\draw[ thick, join=round, cap=round] (0,0+2.6)--(3,1+2.6)--(-1,2+2.6)--(0,0+2.6);
\draw[ thick, join=round, cap=round] (0,0+3.9)--(3,1+3.9)--(-1,2+3.9)--(0,0+3.9);
\draw[ thick, join=round, cap=round] (0,0+5.2)--(3,1+5.2)--(-1,2+5.2)--(0,0+5.2);
\draw[ thick, join=round, cap=round] (0,0+1.3)--(3,1+1.3)--(-1,2+1.3)--(0,0+1.3);
\draw[ thick, join=round, cap=round] (3,1)--(5,1);
\draw[ thick, join=round, cap=round] (3,1+1.3)--(5,1+1.3);
\draw[ thick, join=round, cap=round] (3,1+2.6)--(5,1+2.6);
\draw[ thick, join=round, cap=round] (3,1+3.9)--(5,1+3.9);
\draw[ thick, join=round, cap=round] (3,1+5.2)--(5,1+5.2);
\draw[ thick, join=round, cap=round] (3,1)--(3,8);
\draw[ thick, join=round, cap=round] (5,1)--(5,8);
\draw[ thick, join=round, cap=round] (0,0)--(0,7.5);
\draw[ thick, join=round, cap=round] (-1,2)--(-1,8.5);
\draw[line width = 1.3mm, join=round, cap=round, red] (3,3.6)--(3,3.6+1.3)--(-1,4.6+1.3);
%%%
%%%
\draw[blue, fill =blue]  (0,0) circle (10pt);
\draw[line width = 1.3mm, join=round, cap=round, red] (0,0)--(3,1)--(3,2.3)--(5,2.3)--(5,3.6)--(3,3.6);
\draw[ thick, join=round, cap=round] (0,3.9+1.3)--(-1,5.9+1.3);
\draw[ thick, join=round, cap=round] (0,3.9+1.3)--(0,5.2+1.3);
\draw[line width = 2.1mm, join=round, cap=round, white, shorten >=0.5cm] (0,3.9+1.3)--(3,4.9+1.3);
\draw[line width = 1.3mm, join=round, cap=round, ->, red] (-1,4.6+1.3)--(0,2.6+1.3)--(0,3.9+1.3)--(3,4.9+1.3);
\draw[red, fill =red]  (0,0) circle (7pt);
\draw[red, fill =red]  (3,1) circle (7pt);
\draw[red, fill =red]  (3,2.3) circle (7pt);
\draw[red, fill =red]  (3,4.9) circle (7pt);
\draw[red, fill =red]  (5,2.3) circle (7pt);
\draw[red, fill =red]  (5,3.6) circle (7pt);
\draw[red, fill =red]  (3,3.6) circle (7pt);
\draw[red, fill =red]  (-1,4.6+1.3) circle (7pt);
%\draw[red, fill =red]  (0,2.6) circle (7pt);
\draw[red, fill =red]  (0,3.9) circle (7pt);
\draw[red, fill =red]  (0,3.9+1.3) circle (7pt);
       \node[] at (1.5, 8){ $\vdots$};      
\end{tikzpicture}
}
\;\;\;
\xrightarrow{\;\;f\;\;\;}
\;\;\;
\hackcenter{
\begin{tikzpicture}[scale=0.6]
\draw[black, fill =black]  (0,0) circle (3pt);
\draw[black, fill =black]  (3,1) circle (3pt);
\draw[black, fill =black]  (-1,2) circle (3pt);
\draw[black, fill =black]  (5,1) circle (3pt);
\draw[ thick, join=round, cap=round] (0,0)--(3,1)--(-1,2)--(0,0);
\draw[ thick, join=round, cap=round] (3,1)--(5,1);
  \node[above] at (0.5,0.5){ $\scriptstyle 0$};     
    \node[above] at (4,-0.1){ $\scriptstyle 1$};    
     \node[above] at (4,1.4){ $\scriptstyle 2$};   
     \node[above] at (1,1.8){ $\scriptstyle 3$};   
       \node[above] at (-1.4,0.6){ $\scriptstyle 3$};  
             \node[above] at (1,-1){ $\scriptstyle 4$};  
%%%
%%%
 \draw[ ultra thick, join=round, cap=round, red] (0,0) .. controls ++(.25,0.5) and ++(-.45,0.15) ..(3,1);
  \draw[ ultra thick, join=round, cap=round, red] (3,1) .. controls ++(.45,-0.45) and ++(-.45,-0.45) .. (5,1);
    \draw[ ultra thick, join=round, cap=round, red] (5,1) .. controls ++(-.45,0.45) and ++(+.45,0.45) .. (3,1.2);
       \draw[ ultra thick, join=round, cap=round, red] (3,1.2) .. controls ++(-.45,0.45) and ++(+.45,0.25) .. (-1,2);
        \draw[ ultra thick, join=round, cap=round, red] (-1,2) .. controls ++(-.45,-0.45) and ++(-.45,0.25) .. (-0.2,-0.2);
          \draw[ ultra thick, join=round, cap=round, red, ->] (-0.2,-0.2) .. controls ++(.45,-0.45) and ++(-.55,-0.5) .. (3,0.7);
\end{tikzpicture}
}
\end{align*}
\caption{A configuration \(\bfx \in \mathcal{V}_{G,b,10}\) and its image \(f(\bfx) \in \textup{Tab}_{G,b,10}\).}
\label{fig:configtotab}
\end{figure}

\subsection{Connecting cubical complexes}
In this section we establish the first main result of the paper, describing an isomorphism between the cubical complexes \(\mathcal{X}(\IPGbl)\) and \(\mathcal{S}_{G,b,\ell}\). This isomorphism shows that \(\mathcal{S}_{G,b,\ell}\) is a CAT(0) complex, and yields an algorithm for optimally reconfiguring the robot arm, as detailed in \S\ref{mainres}.

Let \(\bfx \in \mathcal{V}_{G,b,\ell}\). For \(i \in [1,\#\overline \bfx]\), define \(\bfsig^{(\bfx, i)}\) to be the length \(i\) prefix path \(\overline{\bfx}_1 \cdots \overline{\bfx}_i\) in \(\overline \bfx\), and \(a^{(\bfx, i)} = L_\bfx(i) - n_{\bfsig^{(\bfx,i)}}+1\).

\begin{Lemma}\label{0skel}
There is a well-defined bijection
\begin{align}\label{Sigmap}
\Sigma:\mathcal{V}_{G,b,\ell} \to  \mathcal{L}_\textup{con}(\IPGbl),
\qquad
\bfx \mapsto
\mathcal{I} \{
\la \bfsig^{(\bfx,1)}, a^{(\bfx,1)} \ra,
\ldots,
\la \bfsig^{(\bfx,\#\overline \bfx)}, a^{(\bfx,\#\overline \bfx)} \ra
\}
 \end{align}
\end{Lemma}
\begin{proof}
For a \((G,b)\)-path \(\bfq\), let \(\mathcal{V}_{G,b,\ell}^{\preceq \bfq}\) be the full subgraph of \(\mathcal{V}_{G,b,\ell}\) consisting of vertices \(\bfx\) such that \(\overline \bfx \preceq^\textup{pre} \bfq\). Let \(\mathcal{L}_\textup{con}(\IPGbl)^{\preceq \bfq}\) be the set of all consistent lower sets \(\mu \in \mathcal{L}_\textup{con}(\IPGbl)\) such that \(\bfp \preceq^\textup{pre} \bfq\) for all \(\la \bfp, a \ra \in \mu\). Every \(\mu \in \mathcal{L}_\textup{con}(\IPGbl)\) belongs to some such set; indeed, if \(\bfq\) is of maximal length such that \(\la \bfq, b \ra \in \mu\), then every \(\la \bfp, a \ra \in \mu\) must have \(\bfp \preceq^\textup{pre} \bfq\) since \(\mu\) is consistent, and thus \(\mu \in \mathcal{L}_\textup{con}(\IPGbl)^{\preceq \bfq}\). We thus a bijection \(\mathcal{L}(\IPGbl^{\preceq \bfq}) \to \mathcal{L}_\textup{con}(\IPGbl)^{\preceq \bfq}\) given by the identity on lower sets, and we have decompositions:
\begin{align}\label{decomps}
\mathcal{V}_{G,b,\ell} = \bigcup_{\bfq} \mathcal{V}_{G,b,\ell}^{\preceq \bfq}, \qquad
\mathcal{L}_\textup{con}(\IPGbl) = \bigcup_{\bfq} \mathcal{L}_\textup{con}(\IPGbl)^{\preceq \bfq}.
\end{align}

 We have a chain of set bijections:
\begin{align}\label{bijchain}
\mathcal{V}^{\preceq \bfq}_{G,b,\ell} \xrightarrow{f^\bfq} \Tab^{\preceq \bfq} \xrightarrow{g^\bfq} \Tabqinf \xrightarrow{B} \mathcal{L}(J(\Tabqinf)) \xrightarrow{\widehat{(h^\bfq \circ \tau)^{-1}}} \mathcal{L}(\IPGbl^{\preceq \bfq}) \xrightarrow{\textup{id}} \mathcal{L}_\textup{con}(\IPGbl)^{\preceq \bfq},
\end{align}
where \(f^\bfq:= f|_{\mathcal{V}^{\preceq \bfq}_{G,b,\ell}}\) is a bijection by Lemma~\ref{VTabbij}, \(g^\bfq\) is the bijection of \S\ref{exttab}, \(B\) is the bijection of Theorem~\ref{Birk}, and \(\widehat{(h^\bfq \circ \tau)^{-1}}\) is the extension of the poset bijection \((h^\bfq \circ \tau)^{-1}: J(\Tabqinf) \to \IPGbl\) (see Lemma~\ref{exttab}) to the lattices of lower sets.

In view of (\ref{decomps}), to verify the lemma statement, it is enough to check that the composition of bijections (\ref{bijchain}) agrees with the restriction of (\ref{Sigmap}) to \(\mathcal{V}_{G,b,\ell}^{\preceq \bfq}\). Let \(\bfx \in \mathcal{V}_{G,b,\ell}^{\preceq \bfq}\). Then, following \(\bfx\) through the maps in (\ref{bijchain}), we have:
\begin{align*}
\bfx &\xmapsto{f^\bfq} (\overline{\bfx}, L_\bfx) \xmapsto{g^\bfq} L_\bfx^{\bfq, \infty} \xmapsto{B} \{ Q \in J(\Tabqinf) \mid Q \preceq^{\bfq, \infty} L_\bfx^{\bfq, \infty}\}\\
&=\{ Q \in J(\Tabqinf) \mid Q(r) \geq L_\bfx^{\bfq, \infty}(r) \textup{ for } r \in [1,\#\bfq]\}\\
&=\{ L_{ \la \bfp, c \ra}^{\bfq, \infty} \mid \la \bfp, c \ra \in  \IPGbl^{\preceq \bfq}, 
L_{ \la \bfp, c \ra}^{\bfq, \infty}(r)  \geq L_\bfx^{\bfq, \infty}(r) \textup{ for } r \in [1,\#\bfq]\}\\
&=\{L_{ \la  \bfsig^{(\bfx, j)}, c \ra}^{\bfq, \infty} \mid j \in [1, \#\bfx], L_{ \la  \bfsig^{(\bfx, j)}, c \ra}^{\bfq, \infty}(r)  \geq L_\bfx^{\bfq, \infty}(r) \textup{ for } r \in [1,j] \}\\
&=\{L_{ \la  \bfsig^{(\bfx, j)}, c \ra}^{\bfq, \infty} \mid j \in [1, \#\bfx], L_{ \la  \bfsig^{(\bfx, j)}, c \ra}^{\bfq, \infty}(j)  \geq L_\bfx^{\bfq, \infty}(j) \}\\
&=\{L_{ \la  \bfsig^{(\bfx, j)}, c \ra}^{\bfq, \infty} \mid j \in [1, \#\bfx], c   \geq L_\bfx(j) - n_{\bfsig^{(\bfx, j)}}  + 1 \}\\
&\xmapsto{\widehat{(h^\bfq \circ \tau)^{-1}}} 
\{ \la \bfsig^{(\bfx, j)}, c \ra \mid j \in [1, \#\bfx], c   \geq a^{(\bfx, j)} \}\\
&=\{ \la \bfsig^{(\bfx, j)}, c \ra \mid j \in [1, \#\bfx], c   \geq a^{(\bfx, j)} \}\\
&=\mathcal{I} \{
\la \bfsig^{(\bfx,1)}, a^{(\bfx,1)} \ra,
\ldots,
\la \bfsig^{(\bfx,\#\overline \bfx)}, a^{(\bfx,\#\overline \bfx)} \ra
\}\\
&= \Sigma(\bfx),
\end{align*}
as desired.
\end{proof}

\begin{Example}
In Figure~\ref{fig:configtoideal} we show the lower set \(\Sigma(\bfx)\) corresponding to the configuration \(\bfx\) of Figure~\ref{fig:configtotab}.  
\end{Example}

\begin{figure}[h]
\begin{align*}
\mathcal{I}\left \{
\begin{array}{ccc}
\hackcenter{
\begin{tikzpicture}[scale=0.6]
\draw[black, fill =black]  (0,0) circle (3pt);
\draw[black, fill =black]  (3,1) circle (3pt);
\draw[black, fill =black]  (-1,2) circle (3pt);
\draw[black, fill =black]  (5,1) circle (3pt);
\draw[ thick, join=round, cap=round] (0,0)--(3,1)--(-1,2)--(0,0);
\draw[ thick, join=round, cap=round] (3,1)--(5,1);
  \node[above] at (0.5,0.5){ $\raisebox{.5pt}{\textcircled{\raisebox{-.9pt} {0}}}$};   
%%%
%%%
 \draw[ ultra thick, join=round, cap=round, red, ->] (0,0) .. controls ++(.25,0.5) and ++(-.45,0.05) ..(2.6,1);
\end{tikzpicture}
}
&
\hackcenter{
\begin{tikzpicture}[scale=0.6]
\draw[black, fill =black]  (0,0) circle (3pt);
\draw[black, fill =black]  (3,1) circle (3pt);
\draw[black, fill =black]  (-1,2) circle (3pt);
\draw[black, fill =black]  (5,1) circle (3pt);
\draw[ thick, join=round, cap=round] (0,0)--(3,1)--(-1,2)--(0,0);
\draw[ thick, join=round, cap=round] (3,1)--(5,1);
  \node[above] at (0.5,0.5){ $\raisebox{.5pt}{\textcircled{\raisebox{-.9pt} {1}}}$};   
%%%
%%%
 \draw[ ultra thick, join=round, cap=round, red] (0,0) .. controls ++(.25,0.5) and ++(-.45,0.15) ..(3,1);
  \draw[ ultra thick, join=round, cap=round, red,->] (3,1) .. controls ++(.45,-0.45) and ++(-.45,-0.45) .. (4.8,0.9);
\end{tikzpicture}
}
&
\hackcenter{
\begin{tikzpicture}[scale=0.6]
\draw[black, fill =black]  (0,0) circle (3pt);
\draw[black, fill =black]  (3,1) circle (3pt);
\draw[black, fill =black]  (-1,2) circle (3pt);
\draw[black, fill =black]  (5,1) circle (3pt);
\draw[ thick, join=round, cap=round] (0,0)--(3,1)--(-1,2)--(0,0);
\draw[ thick, join=round, cap=round] (3,1)--(5,1);
  \node[above] at (0.5,0.5){ $\raisebox{.5pt}{\textcircled{\raisebox{-.9pt} {1}}}$};   
%%%
%%%
 \draw[ ultra thick, join=round, cap=round, red] (0,0) .. controls ++(.25,0.5) and ++(-.45,0.15) ..(3,1);
  \draw[ ultra thick, join=round, cap=round, red] (3,1) .. controls ++(.45,-0.45) and ++(-.45,-0.45) .. (5,1);
    \draw[ ultra thick, join=round, cap=round, red, ->] (5,1) .. controls ++(-.45,0.45) and ++(+.45,0.45) .. (3,1.2);
\end{tikzpicture}
}
\\
\hackcenter{
\begin{tikzpicture}[scale=0.6]
\draw[black, fill =black]  (0,0) circle (3pt);
\draw[black, fill =black]  (3,1) circle (3pt);
\draw[black, fill =black]  (-1,2) circle (3pt);
\draw[black, fill =black]  (5,1) circle (3pt);
\draw[ thick, join=round, cap=round] (0,0)--(3,1)--(-1,2)--(0,0);
\draw[ thick, join=round, cap=round] (3,1)--(5,1);
  \node[above] at (0.5,0.5){ $\raisebox{.5pt}{\textcircled{\raisebox{-.9pt} {2}}}$};   
%%%
%%%
 \draw[ ultra thick, join=round, cap=round, red] (0,0) .. controls ++(.25,0.5) and ++(-.45,0.15) ..(3,1);
  \draw[ ultra thick, join=round, cap=round, red] (3,1) .. controls ++(.45,-0.45) and ++(-.45,-0.45) .. (5,1);
    \draw[ ultra thick, join=round, cap=round, red] (5,1) .. controls ++(-.45,0.45) and ++(+.45,0.45) .. (3,1.2);
       \draw[ ultra thick, join=round, cap=round, red,->] (3,1.2) .. controls ++(-.45,0.45) and ++(+.45,0.25) .. (-0.8,2.1);
\end{tikzpicture}
}
&
\hackcenter{
\begin{tikzpicture}[scale=0.6]
\draw[black, fill =black]  (0,0) circle (3pt);
\draw[black, fill =black]  (3,1) circle (3pt);
\draw[black, fill =black]  (-1,2) circle (3pt);
\draw[black, fill =black]  (5,1) circle (3pt);
\draw[ thick, join=round, cap=round] (0,0)--(3,1)--(-1,2)--(0,0);
\draw[ thick, join=round, cap=round] (3,1)--(5,1);
  \node[above] at (0.5,0.5){ $\raisebox{.5pt}{\textcircled{\raisebox{-.9pt} {2}}}$};   
%%%
%%%
 \draw[ ultra thick, join=round, cap=round, red] (0,0) .. controls ++(.25,0.5) and ++(-.45,0.15) ..(3,1);
  \draw[ ultra thick, join=round, cap=round, red] (3,1) .. controls ++(.45,-0.45) and ++(-.45,-0.45) .. (5,1);
    \draw[ ultra thick, join=round, cap=round, red] (5,1) .. controls ++(-.45,0.45) and ++(+.45,0.45) .. (3,1.2);
       \draw[ ultra thick, join=round, cap=round, red] (3,1.2) .. controls ++(-.45,0.45) and ++(+.45,0.25) .. (-1,2);
        \draw[ ultra thick, join=round, cap=round, red,->] (-1,2) .. controls ++(-.45,-0.45) and ++(-.45,0.25) .. (-0.3,0.2);
\end{tikzpicture}
}
&
\hackcenter{
\begin{tikzpicture}[scale=0.6]
\draw[black, fill =black]  (0,0) circle (3pt);
\draw[black, fill =black]  (3,1) circle (3pt);
\draw[black, fill =black]  (-1,2) circle (3pt);
\draw[black, fill =black]  (5,1) circle (3pt);
\draw[ thick, join=round, cap=round] (0,0)--(3,1)--(-1,2)--(0,0);
\draw[ thick, join=round, cap=round] (3,1)--(5,1);
  \node[above] at (0.5,0.5){ $\raisebox{.5pt}{\textcircled{\raisebox{-.9pt} {2}}}$};   
%%%
%%%
 \draw[ ultra thick, join=round, cap=round, red] (0,0) .. controls ++(.25,0.5) and ++(-.45,0.15) ..(3,1);
  \draw[ ultra thick, join=round, cap=round, red] (3,1) .. controls ++(.45,-0.45) and ++(-.45,-0.45) .. (5,1);
    \draw[ ultra thick, join=round, cap=round, red] (5,1) .. controls ++(-.45,0.45) and ++(+.45,0.45) .. (3,1.2);
       \draw[ ultra thick, join=round, cap=round, red] (3,1.2) .. controls ++(-.45,0.45) and ++(+.45,0.25) .. (-1,2);
        \draw[ ultra thick, join=round, cap=round, red] (-1,2) .. controls ++(-.45,-0.45) and ++(-.45,0.25) .. (-0.2,-0.2);
          \draw[ ultra thick, join=round, cap=round, red, ->] (-0.2,-0.2) .. controls ++(.45,-0.45) and ++(-.55,-0.5) .. (3,0.7);
\end{tikzpicture}
}
\end{array}
\right \}
\end{align*}
\caption{The consistent lower set \(\Sigma(\bfx) \in \mathcal{L}_{\textup{con}}(\textup{IP}_{G,b,10})\) corresponding to the configuration \(\bfx \in \mathcal{V}_{G,b,10}\) in Figure~\ref{fig:configtotab}. }
\label{fig:configtoideal}
\end{figure}
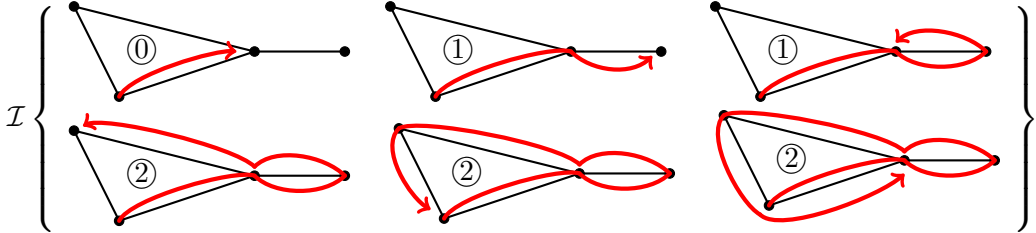

\begin{Theorem}\label{cubeisom}
The bijection \(\Sigma:\mathcal{V}_{G,b,\ell} \to  \mathcal{L}_\textup{con}(\IPGbl)\) extends to an isomorphism of cubical complexes \(\hat \Sigma: \mathcal{S}_{G,b,\ell} \to \mathcal{X}(\IPGbl)\).
\end{Theorem}
\begin{proof}
The map \(\Sigma\) is a bijection on the 0-skeletons of the two cubical complexes by Lemma~\ref{0skel}. We now verify that \(\Sigma\) induces a bijection of \(k\)-cubes in these spaces.

Recalling \S\ref{robcube}, let \(\bfY\) be a set of \(2^k\) configurations in \(\mathcal{V}_{G,b,\ell}\) such that there exists \(k\) commutative moves \(\mathscr{A}\) and \(\bfx \in \textup{supp}(\mathscr{A})\) with \(\bfY = \{\mathscr{S} \bfx \mid \mathscr{S} \subseteq \mathscr{A}\}\). We may write \(\mathscr{A} = \mathscr{A}^1 \sqcup \mathscr{A}^{-1}\), where \(\mathscr{A}^1\) and \(\mathscr{A}^{-1}\) are the set of upward and downward moves in \(\mathscr{A}\), respectively. Now let \(\mathscr{A}^\textup{rev} = \{ \mathsf{S}^{-1} \mid \mathsf{S} \in \mathscr{A}^{-1}\}\). Note that \(\mathscr{A}^\textup{rev} \cap \mathscr{A}^1 = \varnothing\), by the non-self-intersecting property of robotic arm configurations, and the fact that \(\mathscr{A}\) is a commutative set of moves. Taking \(\widetilde{\mathscr{A}} = \mathscr{A}^1 \sqcup \mathscr{A}^\textup{rev}\) and \(\bfy = \mathscr{A}^{-1} \bfx\), it is straightforward to see that \(\widetilde{\mathscr{A}}\) is a set of \(k\) commutative moves, \(\bfy \in \textup{supp}(\widetilde{\mathscr{A}})\), and \(\bfY = \{\mathscr{S} \bfy \mid \mathscr{S} \subseteq \widetilde{\mathscr{A}}\}\). Therefore all \(k\)-cubes in \(\mathcal{S}_{G,b,\ell}\) may be written in the form \([\mathscr{A}; \bfx]\), where \(\mathscr{A}\) is a set of \(k\) commutative {\em upward} moves, and \(\bfx \in \textup{supp}(\mathscr{A})\). In fact, one can note from considering the support of upward moves, that if \(\mathscr{A}\) is a set of upward moves, and there exists some \(\bfx \in \textup{supp}(\mathscr{A})\), that \(\mathscr{A}\) is necessarily a commutative set of moves.

Thus to show that \(\Sigma\) induces a bijection of cubical complexes, it will suffice to show that there is a bijection \(\chi\) between the set \(\mathscr{L}_\bfx^1\) of all upward moves on \(\bfx\) and the set \(\textup{Max}(\Sigma(\bfx))\) of maximal elements in \(\Sigma(\bfx)\), with \(\Sigma(\mathscr{S} \bfx) = \Sigma(\bfx) \backslash \{\chi(\mathsf{S}) \mid \mathsf{S} \in \mathscr{S}\}\) for all \(\mathscr{S} \subseteq \mathscr{L}_\bfx\).

By the definitions of upward moves in \S\ref{movecat}, each \(\mathsf{L} \in \mathscr{L}_\bfx^1\) must involve a distinct horizontal segment in \(\bfx\). Thus we may enumerate \(\mathscr{L}_\bfx^1 = \{\mathsf{L}_1, \ldots, \mathsf{L}_k\}\), and choose an indexing \(\{i_1 < \cdots < i_k\} \subseteq [1, \#\overline{\bfx}]\) such that \(\mathsf{L}_t\)  involves the \(i_t\)th horizontal segment in \(\bfx\). 
We thus define a map \(\chi: \mathscr{L}_\bfx^1 \to \Sigma(\bfx)\) by setting \(\chi(\mathsf{L}_t) = \la \bfsig^{( \bfx, i_t )}, a^{(\bfx,i_t)}\ra\). We will show that \(\chi\) satisfies the properties in the above paragraph.

Assume that \(\mathsf{M}\) is a legal upward corner move for \(\bfy \in \mathcal{V}_{G,b,\ell}\). Then \(\mathsf{M}\) has the effect of raising, say, the \(m\)th horizontal segment in \(\bfy\) by one, and leaves other horizontal segments unchanged. Thus, in consideration of (\ref{Sigmap}), \(\Sigma(\mathsf{M} \bfy) = \Sigma(\bfy)\backslash \{\la \bfsig^{(\bfy, m)}, a^{(\bfy, m)}\ra \}\). Assume instead that \(\mathsf{M}\) is a legal upward tail move for \(\bfy\). Then \(\mathsf{M}\) has the effect of converting the last (\(\#\overline{\bfy}\)th) horizontal segment in \(\bfy\) into a vertical segment. Thus, in consideration of (\ref{Sigmap}), \(\Sigma(\mathsf{M} \bfy) = \Sigma(\bfy)\backslash \{\la \bfsig^{(\bfy, \#\overline{\bfy})}, a^{(\bfy, \#\overline{\bfy})}\ra \}\). 
Iteratively applying this argument, it follows then that \(\Sigma(\mathscr{S} \bfx) = \Sigma(\bfx) \backslash \{\chi(\mathsf{S}) \mid \mathsf{S} \in \mathscr{S}\}\) for all \(\mathscr{S} \subseteq \mathscr{L}_\bfx\). Moreover, since the image \(\Sigma(\mathsf{L}_t \bfx) = \Sigma(\bfx) \backslash \{\chi(\mathsf{L}_t)\}\) is a consistent lower set by Lemma~\ref{0skel}, it follows that \(\chi(\mathsf{L}_t)\) is a maximal element in \(\Sigma(\bfx)\) for all \(t \in [1, k]\).

It also follows, since each \(\mathsf{L} \in \mathscr{L}_\bfx^1\) involves distinct horizontal segments in \(\bfx\), that \(\chi\) as constructed is injective. Thus it remains to show that \(\chi\) is surjective to complete the proof. In view of (\ref{Sigmap}), every maximal element of \(\Sigma(\bfx)\) is of the form \(\la \bfsig^{(\bfx, j)}, a^{(\bfx, j)} \ra\) for some \(j \in [1, \#\bfx]\). Fix such a maximal element  \(\la \bfsig^{(\bfx, j)}, a^{(\bfx, j)} \ra\). Since it is maximal, we have by Lemma~\ref{0skel} that \(\Sigma(\bfx) \backslash \{\la \bfsig^{(\bfx, j)}, a^{(\bfx, j)} \ra\} = \Sigma(\bfy)\) for some \(\bfy \in \mathcal{V}_{G,b,\ell}\). If \(\la \bfsig^{(\bfx, j)}, a^{(\bfx, j)}+1 \ra \in \Sigma(\bfx)\), it follows then from (\ref{Sigmap}) that \(\bfy\) has the same number of horizontal segments as \(\bfx\), and each is in the same position, except that the \(j\)th horizontal segment in \(\bfy\) is one unit higher than the \(j\)th horizontal segment in \(\bfx\). Thus \(\bfy\) is achieved from \(\bfx\) by an upward corner move, and so \(\la \bfsig^{(\bfx, j)}, a^{(\bfx, j)} \ra = \chi(\mathsf{L})\) for some \(\mathsf{L} \in \mathscr{L}_\bfx^1\).

 On the other hand, assume that \(\la \bfsig^{(\bfx, j)}, a^{(\bfx, j)}+1 \ra \notin \Sigma(\bfx)\). Then 
\begin{align*}
a^{(\bfx,j)} + 1 = L_\bfx(j) - n_{\bfsig^{(\bfx,j)}} + 2 > \ell + 1 - \#\bfsig^{(\bfx,j)} - n_{\bfsig^{(\bfx,j)}},
\end{align*}
so \(L_\bfx(j) + 1+j > \ell \). Since \(L_{\bfx}(j) \leq \ell - j\), it follows that \(j = \ell - L_{\bfx}(j)\), which implies that \(j = \#\overline{\bfx}\), i.e., the \(j\)th horizontal segment in \(\bfx\) is the last segment in \(\bfx\). It follows then from (\ref{Sigmap}) that \(\bfy\) has \(j-1\) horizontal segments, and each is in the same position as the first \(j-1\) horizontal segments of \(\bfx\). Thus \(\bfy\) is achieved from \(\bfx\) by an upward tail move, and so again \(\la \bfsig^{(\bfx, j)}, a^{(\bfx, j)} \ra = \chi(\mathsf{L})\) for some \(\mathsf{L} \in \mathscr{L}_\bfx^1\). Thus \(\chi\) is surjective, as desired.
\end{proof}

\begin{Corollary}\label{optimove}
There exists an explicit algorithm for optimally moving the robotic arm \(\mathcal{R}_{G,b,\ell}\) between two configurations:
\begin{enumerate}
\item[(I)] in a minimal number of total local moves, or; 
\item [(II)] in a minimal time, given that independent local moves can be performed simultaneously.
\end{enumerate}
\end{Corollary}
\begin{proof}
Given that in Theorem~\ref{cubeisom} we have a constructed a PIP \(\IPGbl\) explicitly associated with the complex \(\mathcal{S}_{G,b,\ell}\), one may follow the algorithm of \cite[\S6]{ABY} to optimally move between configurations, either in a move-minimizing fashion (\cite[\S6.3]{ABY}), or in a time-minimizing fashion (\cite[\S6.4, 6.5]{ABY}).
\end{proof}

\subsection{Implementation}
In the case where \(G=C_4\) is the cycle graph with 4 vertices, we have implemented an algorithm such as described in (II) in Python, which, using simultaneous independent moves when possible, moves the robotic arm between different configurations in a minimal time period. This program is available to the reader by request.

\section{Distance and diameter in the robotic arm transition graph}

In this final section, we establish the second main result of the paper, giving a tight bound on the diameter of the transition graph \(\mathcal{T}_{G,b,\ell}\), and an explicit diameter value for certain common and interesting families of graphs \(G\).

For \(\bfx \in \mathcal{V}_{G,b,\ell}\), let \(f(\bfx) = (\overline{\bfx}, L_{\bfx})\) be the associated path tableau, and define a function \(M_{\bfx}: [1, \overline{\bfx}] \to \mathbb{Z}_{\geq 0}\) by
\(
M_{\bfx}(i) = \ell - L_{\bfx}(i)  - i+ 1.
\)

\begin{Proposition}\label{distx}
For \(\bfx, \bfy \in \mathcal{V}_{G,b,\ell}\), the distance between \(\bfx\) and \(\bfy\) in the transition graph \(\mathcal{T}_{G,b,\ell}\) is given by 
\begin{align*}
d(\bfx, \bfx') = \sum_{j = r+1}^{\# \overline{\bfx}} M_\bfx(j) +  \sum_{j = r+1}^{\# \overline{\bfy}} M_\bfy(j) + \sum_{j=1}^r |M_{\bfx}(j) - M_{\bfy}(j)|,
\end{align*}
where \(r\) is maximal such that \(\overline{\bfx}_i = \overline{\bfy}_i\) for \(i \leq r\).
\end{Proposition}
\begin{proof}
In \cite[Proposition 6.2]{ABCG} an analogous statement for the transition graph of a robotic arm in a tunnel is proved using the combinatorics of coral tableaux. Translating from the combinatorial setup of that paper to this one yields a proof of the above statement in the case \(G = A_n\) using the combinatorics of path tableaux, which, with minor alteration, may be adapted to a proof for arbitrary \(G\).
\end{proof}

\begin{Theorem}\label{diamthm}
Assume that \(n=|V_G| \). Then we have
\begin{align}\label{diamineq}
\textup{diam}(\mathcal{T}_{G,b,\ell}) \leq 2 \left\lfloor \frac{(n - 1)(\ell + 1)^2}{2n}\right \rfloor,
\end{align}
where equality is achieved if there exist two cycle-free \((G,b)\)-paths \(\bfp, \bfp'\) of length \(\min\{\ell, n-1\}\) with \(\bfp_1 \neq \bfp_1'\).
\end{Theorem}
\begin{proof}
Let \(m\) be maximal such that \(\lfloor (m-1)/(n-1) \rfloor + m \leq \ell\). Define a function 
\begin{align*}
\omega(\ell,n) = \sum_{j=1}^m \ell -\left \lfloor \frac{j-1}{n-1} \right \rfloor - j + 1.
\end{align*}
Note that expanding the sum yields
\begin{align*}
 \ell + (\ell -1 ) + \cdots + (\ell - n+2) + (\ell - n) + \cdots + (\ell - 2n+2) + (\ell - 2n) + \cdots + (\ell - 3n+2) + \cdots 
\end{align*}
with the last term in the sum being 2 if \(n\) divides \(\ell\), and 1 otherwise. Thus \(\omega(\ell,n)\) is the sum over all integers \(t \in [1,\ell]\) omitting those of the form \( t= \ell - kn +1\). In other words, we may also write
\begin{align}\label{omegaform}
\omega(\ell,n) = \sum_{\substack{t \in [1,\ell] \\ t \not\equiv  \ell + 1 \pmod n}} \hspace{-5mm}t.
\end{align}

Keeping \(n\) fixed, we now show by induction on \(\ell\) that
\begin{align}\label{IndClaimA}
\omega(\ell,n) =  \left\lfloor \frac{(n - 1)(\ell + 1)^2}{2n}\right \rfloor.
\end{align}
for all \(\ell,n \geq 0\).
First assume that \( \ell < n\). Then no \(t \in [1,\ell]\) is congruent to \(\ell + 1\) modulo \(n\), so we have 
\begin{align}\label{IndClaimB}
\omega(\ell,n) = 1 + 2 + \cdots + \ell = \frac{\ell(\ell+1)}{2}.
\end{align}
Now note that
\begin{align*}
\frac{(\ell + 1)^2}{(\ell + 3)} < \ell + 1 \leq n,
\end{align*}
so \((\ell + 1)^2 < n(\ell +3)\), which yields \(\ell^2 - n \ell - n + 2 \ell + 1 < 2n\), and therefore
\begin{align*}
   0 &= \floor*{\frac{\ell^2 -\ell n -n + 2\ell +\ell}{2n}} =     \floor*{\frac{\ell n(\ell +1)-(n-1)(\ell +1)^2}{2n}}= 
  \frac{\ell (\ell +1)}{2} -\floor*{\frac{(n-1)(\ell +1)^2}{2n}}.
\end{align*}
Thus (\ref{IndClaimA}) holds by (\ref{IndClaimB}).

Now, for the induction step, we may assume that \(\ell = k+ n\) for some \(k \geq 0\), and that (\ref{IndClaimA}) holds for \(\omega(k,n)\). Then by consideration of (\ref{omegaform}), we have
\begin{align*}
\omega(\ell, n) &= (k+n) + (k+n-1) + \cdots + (k+2) + \omega(k, n)\\
&=(k+n) + (k+n-1) + \cdots + (k+2)+ \left\lfloor \frac{(n - 1)(k + 1)^2}{2n}\right \rfloor\\
&= \frac{(n-1)(2k+n+2)}{2}+\floor*{ \frac{(n-1)(k+1)^2}{2n}}\\
    &= \floor*{\frac{(n-1)(n^2 + 2kn + 2n + k^2 +2k+1)}{2n}}\\
    &= \floor*{\frac{(n-1)(k+n+1)^2}{2n}}\\
    &=  \left\lfloor \frac{(n - 1)(\ell + 1)^2}{2n}\right \rfloor,
\end{align*}
completing the induction step and the proof of (\ref{IndClaimA}).

Now we show that \( d(\bfy, \bfb) \leq \omega(\ell,n)\) for all \(\bfy \in \mathcal{V}_{G,b,\ell}\). Let \(\bfy \in \mathcal{V}_{G,b,ell}\), and let \(f(\bfy) = (\overline{\bfy}, L_{\bfy})\) be the associated \((G,b,\ell)\)-tableau. We note that by Definition~\ref{deftab}(ii), for all \(k \in \mathbb{Z}_{\geq 0}\) we have \(|\{ i \in [1,\#\bfy] \mid L_\bfy(i) = k\} \leq n-1.\) Therefore \(L_{\bfy}(i) \geq \lfloor (i-1)/(n-1) \rfloor\) for all \(i \in [1, \#\bfy]\). Thus
\begin{align*}
\#\overline{\bfy} \leq \ell - L_\bfy(\#\overline{\bfy}) \leq \ell - \lfloor (\#\overline{\bfy}-1)/(n-1) \rfloor 
\end{align*}
by Definition~\ref{deftab}(iii). As \(m\) is maximal such that \(m \leq \ell - \lfloor (m-1)/(n-1) \rfloor\), it follows then that \(\#\overline{\bfy} \leq m\). 
Therefore we have
\begin{align*}
d(\bfy, \bfb) &= \sum_{j=1}^{\#\overline{\bfy}} M_\bfy(j) =  \sum_{j=1}^{\#\overline{\bfy}} \ell - L_\bfy(j) - j + 1 \leq  \sum_{j=1}^{\#\overline{\bfy}} \ell - \left \lfloor \frac{j-1}{n-1} \right \rfloor - j + 1\\
&\leq  \sum_{j=1}^{m} \ell - \left \lfloor \frac{j-1}{n-1} \right \rfloor - j + 1
= \omega(\ell,n).
\end{align*}

Therefore, for any \(\bfy, \bfy' \in \mathcal{V}_{G,b,\ell}\), we have
\begin{align*}
d(\bfy, \bfy') \leq d(\bfy, \bfb) + d(\bfy', \bfb) \leq 2 \omega(\ell,n) = 2 \left\lfloor \frac{(n - 1)(\ell + 1)^2}{2n}\right \rfloor,
\end{align*}
establishing (\ref{diamineq}).

Now assume that there exist two cycle-free \((G,b)\)-paths \(\bfp, \bfp'\) of length \(\min\{\ell, n-1\}\) with \(\bfp_1 \neq \bfp_1'\). If \(\ell < n-1\), then note that \(m = \ell\), and set \(\bfq := \bfp\) and \(\bfq' := \bfp'\). If \(\ell \geq n-1\), define a \((G,b)\)-path \(\bfq\) with \(\#\bfq = m\) by setting \(\bfq\) to be the \(m\)-length prefix path in
\begin{align*}
(\bfp_1 \cdots \bfp_{n-1} \bfp_{n-1} \cdots \bfp_1 \bfp_1)^N,
\end{align*}
where \(N\gg 0\). Informally, \(\bfq\) is the path that consists of continuously tracing the path \(\bfp\) back and forth for \(m\) total steps. Define \(\bfq'\) similarly, using the path \(\bfp'\).

In any case now, setting \(L(i) = \lfloor (i-1)/(n-1) \rfloor\), we check that \((\bfq, L) \in \Tab\) by verifying axioms (i--iii) of Definition~\ref{deftab}. Parts (i),(iii) are clear by the construction of \(L\). Now assume that \(i < j\) and \(q_{i-1} = q_j\) (or \(q'_{i-1} = q'_j\)). Then we have that \(j \geq i-1 +n\), so 
\begin{align*}
L(j) \geq L(i-1+n) = \left \lfloor \frac{i-1+n -1}{n-1} \right\rfloor = \left\lfloor \frac{ i-1}{n-1} + \frac{n-1}{n-1}\right \rfloor =\left \lfloor \frac{i-1}{n-1}\right\rfloor + 1 = L(i) + 1,
\end{align*}
verifying (ii). 

Therefore by Lemma~\ref{VTabbij} there exists \(\bfx, \bfx' \in \mathcal{V}_{G,b,\ell}\) such that \(f(\bfx) = (\bfq, L)\) and \(f(\bfx') = (\bfq', L)\). Then we have
\begin{align*}
\sum_{j=1}^m M_\bfx(j) =  \sum_{j=1}^m \ell - L(j) - j + 1 = \sum_{j=1}^m \ell - \left \lfloor \frac{j-1}{n-1} \right\rfloor - j + 1 = \omega(\ell,n).
\end{align*}
Thus, since \(\bfq_1 \neq \bfq_1'\), by Proposition~\ref{distx} and (\ref{diamineq}), we have
\begin{align*}
2 \omega(\ell,n) \geq \textup{diam}(\mathcal{T}_{G,b,\ell}) \geq d(\bfx, \bfx') = \sum_{j = r+1}^{\# \overline{\bfx}} M_\bfx(j) +  \sum_{j = r+1}^{\# \overline{\bfy}} M_\bfy(j) = 2 \omega(\ell,n),
\end{align*}
so
\begin{align*}
\textup{diam}(\mathcal{T}_{G,b,\ell}) = 2 \left\lfloor \frac{(n - 1)(\ell + 1)^2}{2n}\right \rfloor,
\end{align*}
as desired.
\end{proof}

\begin{Example}
Consider the case of the cycle graph \(C_3\), a fixed vertex \(b \in C_3\), and the robotic arm of length \(\ell = 5\). There are two cycle-free paths of length 2 originating at \(b\); one proceeding clockwise and the other counterclockwise. Thus by Theorem~\ref{diamthm} we expect that
\begin{align*}
\textup{diam}(\mathcal{T}_{C_3, b, 5}) = 2 \left \lfloor \frac{(3-1)(5+1)^2}{2 \cdot 3} \right \rfloor = 24.
\end{align*}
Indeed, in Figure~\ref{fig:fullspace}, we see that the path from the leftmost to rightmost nodes in the 1-skeleton \(\mathcal{T}_{C_3,b,5}\) in the image yields a maximum distance of 24.
\end{Example}

\begin{Remark}
In general, if there do not exist two cycle-free \((G,b)\)-paths \(\bfp, \bfp'\) of length \(\min\{\ell, n-1\}\) with \(\bfp_1 \neq \bfp_1'\), equality in Theorem~\ref{diamthm} is not achieved. For example, in \cite[\S6]{ABCG}, the case of the path graph \(G = A_n\), where \(b\) is an endpoint of the graph, is considered. In this setting the diameter of the transition graph is shown to be strictly less than the bound of Theorem~\ref{diamthm}.
\end{Remark}

\begin{Remark}\label{RemConnex}
With \(\omega(\ell,n)\) defined as in (\ref{IndClaimA}), Theorem~\ref{diamthm} gives us that \(2 \omega(\ell,n)\) is a tight bound on the diameter of the transition graph \(\mathcal{T}_{G,b,\ell}\) when \(G\) has \(n\) vertices. Moreover, this bound is achieved when \(G\) possesses two cycle-free \((G,b)\)-paths \(\bfp, \bfp'\) of length \(\min\{\ell, n-1\}\) with \(\bfp_1 \neq \bfp_1'\), as is the case for large families of graphs including all cycle graphs, complete graphs, and graphs which possess a Hamiltonian circuit. 

We now briefly remark on a connection of \(\omega(\ell,n)\) with other known sequences. The sequence \((\omega(\ell,n))_{\ell = 0}^\infty\) makes an appearance in the OEIS \cite{OEIS} as the sequences
\href{http://oeis.org/A002620}{[A002620]},
\href{http://oeis.org/A000212}{[A000212]},
\href{http://oeis.org/A033436}{[A033436]},
\href{http://oeis.org/A033437}{[A033437]},
\href{http://oeis.org/A033438}{[A033438]},
\href{http://oeis.org/A033439}{[A033439]},
for \(n=2,3,4,5,6, 7\), respectively (excluding the first term in the OEIS sequence). As noted in these OEIS links, the value \(\omega(\ell,n)\) gives the number of edges in the {\em n-partite Turan graph} \(T(\ell +1, n)\); see \cite{Wolf}. We also have \(\omega(\ell,n) = R_{\ell+1}(1,n)\), where \(R_{\ell+1}(1,n)\) is the {\em elliptic troublemaker sequence} as defined in \cite[Definition 23, Proposition 24(viii)]{Stange}. 
\end{Remark}

%%%%%%%%%%%%%%%%%%%%%%%%%%%%%%%%%%%%%%%%

\end{document}